\documentclass{amsart}
\usepackage[leqno]{amsmath}
\usepackage{amsfonts,amssymb,amscd,enumerate,verbatim, calc,color,boldline,xypic, makecell, booktabs}
\usepackage{latexsym,colonequals}
\usepackage{colortbl}
\usepackage{longtable}
\usepackage{upgreek, hyperref}
\usepackage{amsthm}
\usepackage{arydshln}
\usepackage{mathtools}
\usepackage{nicematrix}
\usepackage[pointedenum]{paralist}
\CompileMatrices

\newcommand{\betrow}[6]{\mathtt{#1} & \mathtt{#2} & \mathtt{#3} & \mathtt{#4}& \mathtt{#5} &\mathtt{#6}}
\newcommand{\HH}[2]{\operatorname{H}_{#1}(#2)}
\newcommand{\h}[1]{\-\mbox{-#1}}
\newcommand{\im}{\operatorname {Im}}
\newcommand{\Ann}{\operatorname{Ann}}

\newcommand{\Ker}{\operatorname{Ker}}
\newcommand{\Coker}{\operatorname{Coker}}
\newcommand{\clGH}{\operatorname{GH}}
\newcommand{\D}{\operatorname{D}}
\newcommand{\HF}[1]{\operatorname{H}_{#1}}
\newcommand{\Po}[2]{\operatorname{P}^{#1}_{#2}}
\newcommand{\cone}{\operatorname{Cone}}

\newcommand{\rank}{\operatorname{rank}}
\newcommand{\Hom}[3]{\operatorname{Hom}_{#1}(#2,#3)}
\newcommand{\Tor}[4]{\operatorname{Tor}_{#1}^{#2}(#3,#4){}}
\newcommand{\Ext}[4]{\operatorname{Ext}_{#1}^{#2}(#3,#4){}}

\newcommand{\cha}{\operatorname{char}}

\newcommand{\grade}[2]{\operatorname{grade}_{#1}{#2}}
\newcommand{\xra}{\xrightarrow}

\newcommand{\ee}{\mathsf{e}}
\newcommand{\eg}{\mathsf{g}}
\newcommand{\ff}{\mathsf{f}}
\newcommand{\hh}{\mathsf{h}}
\newcommand{\kk}{\mathsf{k}}
\newcommand{\zz}{\mathbb{Z}}
\newcommand{\A}{\mathsf{A}}

\newcommand{\fD}{\mathfrak{D}}

\newcommand{\mff}{\mathfrak{f}}

\newcommand{\bC}{{\boldsymbol C}}
\newcommand{\bF}{{\boldsymbol F}}
\newcommand{\bG}{{\boldsymbol G}}

\newcommand{\BP}{\mathbb{P}}
\newcommand{\BZ}{\mathbb{Z}}

\newtheorem{theorem}{Theorem}[section]
\newtheorem{corollary}[theorem]{Corollary}
\newtheorem{lemma}[theorem]{Lemma}

\newtheorem{proposition}[theorem]{Proposition}
\newtheorem*{proposition*}{Proposition}

\newtheorem*{maintheorem}{Main Theorem}
\theoremstyle{definition}
\newtheorem{chunk}[theorem]{}

\newtheorem{remark}[theorem]{Remark}
\newtheorem{definition}[theorem]{Definition}

\newtheorem{example}[theorem]{Example}
\newtheorem{setup}[theorem]{Setup}

\numberwithin{equation}{theorem}

\theoremstyle{remark}

\allowdisplaybreaks

\title[Artinian Gorenstein algebras of emb.\ dim.\ 4 and socle deg.\ 3]{Artinian Gorenstein algebras of  embedding dimension four and socle degree three}
\author[P.~M.~Marques]{Pedro Macias Marques}
\address{Departamento de Matem\'{a}tica, Escola de Ci\^{e}ncias e Tecnologia, Centro de Investiga\c{c}\~{a}o em Matem\'{a}tica e Aplica\c{c}\~{o}es, Instituto de Investiga\c{c}\~{a}o e Forma\c{c}\~{a}o Avan\c{c}ada, Universidade de \'{E}vora, Rua Rom\~{a}o Ramalho, 59, P--7000--671 \'{E}vora, Portugal}
\email{pmm@uevora.pt}
\author[O.~Veliche]{Oana Veliche}
\address{Department of Mathematics, Northeastern University, Boston, MA 02115, U.S.A.}
\email{o.veliche@northeastern.edu}
\author{Jerzy Weyman}
\address{Department of Mathematics, Jagiellonian University, Krakow, Poland}
\email{jerzy.weyman@gmail.com}
\thanks{}
\date{\today}
\keywords{Artinian Gorenstein algebra,  Macaulay inverse system, doubling, free resolution, connected sum}
\subjclass[2020]{Primary 13C05. Secondary 13H10, 13A02, 13D02, 13D07}

\begin{document}
\maketitle
\begin{abstract}
We prove that in the polynomial ring $Q=\kk[x,y,z,w]$, with $\kk$ an algebraically closed field of characteristic zero, all Gorenstein homogeneous ideals $I$ such that $(x,y,z,w)^4\subseteq I \subseteq (x,y,z,w)^2$ can be obtained by  \emph{doubling} from a grade three perfect ideal $J\subset I$ such that $Q/J$ is a locally Gorenstein ring. Moreover, a graded minimal free resolution of the \mbox{$Q$-module} $Q/I$  can be completely described in terms of a graded minimal free resolution of the \mbox{$Q$-module} $Q/J$ and a homogeneous embedding of a shift  of the canonical module $\omega_{Q/J}$  into $Q/J$. 
\end{abstract}
\tableofcontents

\section{Introduction}

This paper is devoted to Artinian Gorenstein graded algebras of embedding dimension four and socle degree three. More precisely, we consider rings of the form  $Q/I$, where $Q=\kk[x,y,z,w]$ with $\kk$ an algebraically closed field of characteristic zero, and $I$ is a grade four homogeneous Gorenstein ideal such that the graded quotient ring $Q/I$ has the Hilbert function $\HF{Q/I}(t)=1+4t+4t^2+t^3$. 

The main goal of this paper is to understand the structure of  all such Gorenstein ideals $I$, together with the structure of a graded minimal free resolution $\bF$ of each  $Q$-module $Q/I$. Our main result is that for each ideal $I$ there exists a grade three perfect homogeneous ideal $J\subset I$ such that the ideal $I$ is obtained by a \emph{doubling} of  $J$. In each case, the resolution ${\bF}$ can be obtained by a \emph{doubling} of a graded minimal free resolution ${\bG}$ of the $Q$-module $Q/J$. The \emph{resolution format} of the ideal $J$ is defined as $\mff_J=(1,m,m+t-1,t)$, where $m=\rank_QG_1$ is the number of generators of $J$ and $t=\rank_Q G_3$ is the type of the quotient Cohen-Macaulay ring $Q/J$. Since $Q/I$ is an Artinian Gorenstein ring, the \emph{resolution format} of the ideal $I$ is  given by $\mff_I=(1,n,2n-2,n,1)$ where $n=\rank_QF_1=\rank_QF_2$ and $2n-2=\rank_QF_3$.
Fusing Corollary \ref{cor: Iquadratic}(a) and Theorems \ref{mu7 structure} and \ref{mu9 structure} we obtain the following:

\vspace{0.5cm}

\begin{maintheorem}
Let $\kk$ be an algebraically closed field with $\cha \kk =0$ and $I$ a homogeneous Gorenstein ideal in the polynomial ring $Q =\kk[x,y,z,w]$ with  $(x,y,z,w)^4\subseteq I \subseteq (x,y,z,w)^2$. Then, there exist a grade three perfect quadratic ideal $J$ such that $J\subset I$ and a homogeneous embedding
$\iota$ of a shift of the canonical module of $Q/J$ into $Q/J$ such that the following sequence is exact:
\[0\rightarrow \omega_{Q/J}(-3)\xra{\iota} Q/J\xra{\pi} Q/I\rightarrow 0,\]
where $\pi$ is the canonical projection.
Moreover, the minimal number of generators of the ideal $I$ is either six, seven, or nine, and the ideals $J$ and $I$ have the following resolution formats, respectively:
\[
\begin{array}{cc}
\mff_J&\mff_I\\
\hline
(1,5,5,1)&(1, 6,10,6,1)\\
(1,5,6,2)&(1, 7,12,7,1)\\
(1,6,8,3)&(1, 9,16,9,1).
\end{array}
\]
\end{maintheorem}

The method of proof is based first on a precise description, up to a linear change of variables,  of a dual generator that corresponds to the Gorenstein ideal $I$, through the Macaulay inverse system.
Next, we divide into three cases depending on the minimal number of generators of the ideal $I$ that could be six, seven, or nine. Finally, we treat separately the cases when $Q/I$ is a connected sum of two Artinian Gorenstein algebras of smaller embedding dimension.

Combining  Corollaries \ref{cor: Iquadratic}(b), \ref{cor: mu7 structure}, and \ref{cor: mu9 structure}, we get the following:

\begin{proposition*}
Let $\kk$ be an algebraically closed field with $\cha \kk =0$ and $I$ a homogeneous Gorenstein ideal in the polynomial ring $Q =\kk[x,y,z,w]$ with  $(x,y,z,w)^4\subseteq I \subseteq (x,y,z,w)^2$. Then, the Betti table of $Q/I$ is one of the following form:
{\scriptsize
      \begin{align*}
        \begin{array}{r|ccccc}
          \betrow{}{0}{1}{2}{3}{4} \\
          \hline
          \betrow{0}{1}{.}{.}{.}{.} \\
          \betrow{1}{.}{6}{5}{.}{.} \\        
          \betrow{2}{.}{.}{5}{6}{.} \\
          \betrow{3}{.}{.}{.}{.}{1} \\
          \hline
          \betrow{total}{1}{6}{10}{6}{1}
        \end{array}
      &&
        \begin{array}{r|ccccc}
          \betrow{}{0}{1}{2}{3}{4} \\
          \hline
          \betrow{0}{1}{.}{.}{.}{.} \\
          \betrow{1}{.}{6}{6}{1}{.} \\        
          \betrow{2}{.}{1}{6}{6}{.} \\
          \betrow{3}{.}{.}{.}{.}{1} \\
          \hline
          \betrow{total}{1}{7}{12}{7}{1}
        \end{array}
      &&
        \begin{array}{r|ccccc}
          \betrow{}{0}{1}{2}{3}{4} \\
          \hline
          \betrow{0}{1}{.}{.}{.}{.} \\
          \betrow{1}{.}{6}{8}{3}{.} \\        
          \betrow{2}{.}{3}{8}{6}{.} \\
          \betrow{3}{.}{.}{.}{.}{1} \\
          \hline
          \betrow{total}{1}{9}{16}{9}{1\rlap{ .}}
        \end{array}
      \end{align*}
    }
\end{proposition*}
This result was independently obtained by Abdallah and Schenk in a paper where they study the weak and the strong Lefschetz properties of Artinian Gorenstein algebras 
\cite[Theorem 1.4]{AS22}.

The paper is organized as follows. In Section \ref{doubling} we define the \emph{doubling} construction for both the  ideal $I$ and for  a graded minimal free resolution $\bF$ of the $Q$-module $Q/I$. In Section \ref{AGsocdeg3} we divide our algebras into two basic classes of quadratic and non\h{quadratic}.
We use a result of Conca, Rossi, and Valla \cite{CRV01} to characterize short non\h{quadratic} algebras as algebras with at least one linear form of rank one. Our restrictions on the field  $\kk$ are mostly introduced due to the fact that a proof of this result is known by the authors only over algebraically closed fields of characteristic zero.  
In Section \ref{Structure6gens} we deal with quadratic algebras. We prove that if $I$ is quadratic, then $Q/I$ is a hyperplane section of a one dimensional Gorestein quotient ring $Q/J$. The corresponding resolution $\bF$ is a tensor product of the Koszul complex in one regular element and a Buchsbaum-Eisenbud resolution $\bG$ of $Q/J$  \cite[Theorem 2.1]{BE77a}. This result was originally proved by El Khoury and Kustin in \cite{EKK17}, over a field of characteristic different from two. 
Our proof is shorter, but it applies only on an algebraically closed field of characteristic zero. It is based on some results of Avramov \cite{Avr89} and \cite{CRV01}; see Theorem \ref{Iquadratic}. In Section \ref{NonquadGorideals} we divide the non\h{quadratic} algebras based on the form of the dual generator into families that can be dealt with. 
Two main cases, of seven and nine generators, are divided into several subcases. 
In Section \ref{Structure7gens} we deal with seven generator cases. We detect a quadratic grade three perfect ideal $J$ whose resolution format is $\mff_J=(1,5,6,2)$, and it is of the type described in Brown \cite{Bro87}, depending on a $5\times 5$ skew\h{symmetric}  matrix.
In Section \ref{Structure9gens} we deal with nine generator cases. The  grade three perfect ideal $J$, whose resolution format is $\mff_J=(1,6,8,3)$, is just the ideal generated by the  quadratic minimal generators of the ideal $I$.
It turns out that a graded minimal free resolution $\bG$ of $Q/J$ is an Eagon\h{Northcott} complex resolving $2\times 2$ minors of a $2\times 4$ matrix; see  \cite{EN62}.

\section{Doubling}
\label{doubling}

In this section, let $S$ be a Gorenstein ring  and $J$ a grade $g$ perfect ideal of $S$. The canonical module of $S/J$, unique up to isomorphism, is $\omega_{S/J}=\Ext{S}{g}{S/J}{S}$.
Its rank is one if and only if $S/J$  is a generically Gorenstein ring; see e.g. \cite[Proposition 3.3.18]{BH93}. Recall that a ring is called \emph{generically Gorestein} if its localization at each minimal prime is a Gorenstein ring.

If $(\bC,\partial^{\bC})$ is a complex, then  $(\bC^*,\partial^{\bC^*})$ denotes  the complex where $(\bC^*)_i=\Hom{S}{C_{-i}}{S}$ and $\partial_i^{\bC^*}=\Hom{S}{\partial_{-i}^\bC}{S}$, and  $\bigl(\Sigma^{j}\bC,\partial^{\Sigma^{j}\bC}\bigr)$ denotes the complex where $(\Sigma^{j}\bC)_i=C_{i-j}$ and $\partial^{\Sigma^{j}\bC}_i=(-1)^{j}\partial_{i-j}^{\bC}$, for all $i,j\in \BZ$.

\begin{definition}
\label{def: doubling}
Let $I$ be a Gorenstein ideal of $S$ of grade $g+1$.
If there exists a grade $g$ perfect ideal $J$ of $S$ such that $J\subset I$  and an embedding $\iota$
of the canonical module $\omega_{S/J}$ of $S/J$ into $S/J$ such that the following sequence is exact
\[
0\rightarrow \omega_{S/J}\xra{\iota} S/J\xra{\pi} S/I\rightarrow 0,
\]
where $\pi$ is the canonical projection, then the ideal $I$ is called a \emph{doubling of  $J$} or a \emph{doubling of  $J$ via $\iota$}. Remark that $S/J$ is a generically Gorenstein ring.

Let $\bG$ be a  free resolution of the $S$-module $S/J$. Then $\Sigma^{-g}\bG^*$ is a free resolution of the canonical module $\omega_{S/J}$. The  embedding $\iota$ extends to a chain homomorphism of complexes $\overline\iota\colon \Sigma^{-g}\bG^*\to \bG$.   The complex $\bF\colon=\cone\overline\iota$ is called a \emph{doubling of $\bG$} or a \emph{doubling of $\bG$ via $\overline\iota$}. 
\end{definition}

\begin{example} 
\label{exp doubling}
If $J$ is a grade $g$ Gorenstein ideal of $S$ and $f$ is a regular element of $S/J$, then $J+(f)$ is a Gorenstein ideal of $S$ of grade $g+1$ that is obtained by doubling of the ideal $J$ via the multiplication map
$S/J\xra{f\cdot} S/J.$
\end{example}

The next two lemmas may be well known, but we include them here for the convenience of the reader.

\begin{lemma} 
\label{exact F}
Assume the setup from Definition \ref{def: doubling}.  The complex $\bF$ is a free resolution of the $S$-module $S/I$. 
\end{lemma}
\begin{proof}
There exists a short exact sequence of complexes induced by the homomorphism of complexes $\overline\iota\colon\Sigma^{-g}\bG^*\to\bG$:
\[
0\to\bG\to\bF\to\Sigma^{1-g}\bG^*\to 0,
\]
which induces a long exact sequence in homology:
\[
\cdots\to \HH{i}{\Sigma^{-g}\bG^*}\xra{\HH{i}{\overline\iota}} \HH{i}{\bG}\to\HH{i}{\bF}\to\HH{i-1}{\Sigma^{-g}\bG^*}\xra{\HH{i-1}{\overline\iota}}\HH{i-1}{\bG}\to\cdots
\]
Recall  that 
$\quad \HH{i}{\bG}=
\begin{cases}
S/J&\text{if}\ \, i=0\\
0&\text{if}\ \, i\not=0
\end{cases}\quad$  and $\quad\HH{i}{\Sigma^{-g}\bG^*}=
\begin{cases}
\omega_{S/J}&\text{if}\ \, i=0\\
0&\text{if}\ \, i\not=0.
\end{cases}$

In the case $i\geq 2$, we immediately get $\HH{i}{\bF}=0.$
By setting $i=1$ in the long exact sequence, we get the following exact sequence:
\[
0\to\HH{1}{\bF}\to\omega_{S/J}\xra{\iota} S/J\to\HH{0}{\bF}\to 0.
\]
Since $\iota$ is injective, we conclude that $\HH{1}{\bF}=0$.
Finally, using that the cokernel of $\iota$ is isomorphic to $S/I$ we get that $\bF$ is a free resolution of the $S$-module $S/I$.
\end{proof}

\begin{remark}  
\label{rmk: doubling res}
Consider the ideals $I$, $J$ of $S$ and  the element $f$ of $S$ as in  Example \ref{exp doubling}. The free resolution $\bF$ of the $S$-module $S/I$ obtained by doubling of a free resolution $\bG$ of the $S$-module $S/J$ and the multiplication map by $f$, already appeared in the work of Kustin and Miller see e.g.\ \cite[Proof of Theorem 3.2]{KM82b}. \emph{Doubling} was a folklore term until it appeared for the first time in print in Laxmi's paper \cite{Lax20}.
\end{remark}

\begin{lemma}
\label{J entries}
Let $J$ be a grade $g$ perfect ideal of a Gorenstein  ring $S$ and $(\bG,\partial^{\bG})$ an $S$\h{free} resolution of the $S$-module $S/J$. Then, using the standartd basis, every element  of  the free module $G_g$ that has all entries in $J$ is in the image of $(\partial_g^{\bG})^*$. 
\end{lemma}

\begin{proof} Since $J$ is a perfect ideal, the complex $\Sigma^{-g}\bG^*$ is an $S$-free resolution of the canonical module $\omega_{S/J}=\Ext{S}{g}{S/J}{S}$, hence  
the following sequence of $S$-modules is exact:
\begin{equation*}
    G_{g-1}^*\xra{(-1)^g(\partial_g^{\bG})^*}G_g^*\xra{\varepsilon}\Ext{S}{g}{S/J}{S}\to 0.
\end{equation*}
Since $J\in\Ann_S\big(\Ext{S}{g}{S/J}{S})$, we have that every vector in $G_g^*$ with entries in $J$  is in $\Ker\varepsilon$, hence it is in the image of $(-1)^g(\partial_g^{\bG})^*$, which is the image of $(\partial_g^{\bG})^*$.
\end{proof}

\chunk 
\label{homogeneous canonical}
In what follows, we use definitions and results from \cite[Section 3.5]{BH93}. 
Let $\kk$ be a field, $Q=\kk[x_1,\dots,x_n]$ a polynomial ring, and $J$ a grade $g$ perfect ideal of $Q$. The canonical module of the ring $Q$ is $\omega_Q = Q(-n)$ and the canonical module of the quotient ring $Q/J$ is $\omega_{Q/J}=\Ext{Q}{g}{Q/J}{Q(-n)}$.
Let
\[
\bF\colon 0\to F_g\to F_{g-1}\to\dots\to F_0\to Q/J\to 0
\]
be a graded  minimal free  resolution of $Q/J$ over $Q$, with $F_g=\oplus_{i\in \BZ} Q(-a_i)$. Then, 
\[
\min\{i\mid (\omega_{Q/J})_i\not = 0\}=n-\max_{i\in \BZ}a_i.
\]
Therefore, if in addition  $Q/J$ is a Gorenstein ring, the canonical module of $Q/J$ is
\begin{equation}
\label{Gor canonical}
  \omega_{Q/J}=Q/J\bigl(-n+\max_{i\in\BZ} a_i\bigr). 
\end{equation}

\begin{remark} 
Let $Q =\kk[x,y,z,w]$ with $\kk$ an algebraically closed field and of characteristic zero.
Main Theorem says in particular that each Artinian Gorenstein homogeneous quotient ring $Q/I$ of embedding dimension four and of socle degree three, is obtained by a  doubling of  a grade three perfect quadratic ideal in $Q$.
\end{remark}

\section{Artinian Gorenstein Algebras of socle degree 3}
\label{AGsocdeg3}

Let $\kk$ be a field, $Q=\kk[x_1,\dots,x_n]$ a polynomial ring, and  $\fD=\kk_{DP}[X_1,\dots,X_n]$  a divided power ring.
The ring $Q$ acts on $\fD$ by contraction, i.e.
\[
x^\alpha\circ X^{[\beta]}=
\begin{cases}
X^{[\beta-\alpha]} & \text{if } \beta\ge\alpha \\
0 & \text{otherwise,}
\end{cases}
\]
where $x^\alpha=x_1^{\alpha_1}\cdots x_n^{\alpha_n}$, $X^{[\beta]}=X_1^{[\beta_1]}\cdots X_n^{[\beta_n]}$, $\alpha,\beta\in\mathbb{N}^n$, and we mean $\beta\ge\alpha$ component\h{wise}. Multiplication in $\fD$ is defined by
\[
X^{[\alpha]}\cdot X^{[\beta]}=\tfrac{(\alpha+\beta)!}{\alpha!\beta!}X^{[\alpha+\beta]},
\]
where ${\alpha!=\alpha_1!\cdots\alpha_n!}$, so ${\tfrac{(\alpha+\beta)!}{\alpha!\beta!}=\tfrac{(\alpha_1+\beta_1)!\cdots(\alpha_n+\beta_n)!}{\alpha_1!\cdots\alpha_n!\beta_1!\cdots\beta_n!}}$. We have the following relation between divided powers and regular powers: ${X^\alpha=\alpha!X^{[\alpha]}}$. 
For further details, see for instance the book of Iarrobino and  Kanev \cite[Appendix~A]{IK99}.

In  this section,  $I$ denotes a homogeneous Gorenstein ideal in $Q$ such that $(x_1,\dots,x_n)^4\subseteq I \subseteq (x_1,\dots,x_n)^2$. The standard-graded ring $Q/I$ is thus an Artinian Gorenstein $\kk$-algebra of embedding dimension $n$, socle degree three, and Hilbert function  $\HF{Q/I}{(t)}=1+nt+nt^2+t^3$, i.e.
\[
Q/I=\kk\oplus (Q/I)_1\oplus (Q/I)_2\oplus (Q/I)_3\oplus (Q/I)_4
\]
with $\rank_\kk (Q/I)_0=\rank_\kk (Q/I)_4=1$ and $\rank_\kk (Q/I)_2=\rank_\kk (Q/I)_3=n.$
By Macaulay \cite{Mac94}, there exists a cubic form $F$ in $\fD$, called \emph{dual generator} of  $I$, unique up to scaling by a non-zero field element, such that \[I=\Ann_Q F.\] 
A linear change of variables in the polynomial $F$ will result in a linear change of variables in the ideal $I$, and all the properties, such as minimal number of generators, being Gorenstein, etc., are preserved.

If $\ell$ is a linear form of the ring $Q/I$, then   the rank of the linear  map  given by the multiplication by~$\ell$
\[{(Q/I)_1\xra{\ell\cdot} (Q/I)_2},\] 
is called  the \emph{rank of $\ell$}. Since $Q/I$ is Gorenstein, all non-zero linear forms have positive rank. The following two complementary cases will allow us to better understand the ideal $I$:
\begin{itemize}
    \item[-] There exists in $Q/I$ a linear form of rank one.
    \item[-] All non-zero linear forms of $Q/I$ have rank at least two.
\end{itemize}

The next result concerns the first case and it distinguishes two subcases, the second being the connected sum subcase. 

\begin{proposition}
\label{rankonelinearformdualgenerator}
Let $\kk$ be a field, $Q=\kk[x_1,\dots,x_n]$ a polynomial ring, and $I$ a homogeneous Gorenstein ideal in $Q$ such that $(x_1,\dots,x_n)^4\subseteq I \subseteq (x_1,\dots,x_n)^2$.
If $Q/I$ admits a linear form of rank one, then,  up to a linear change of variables, there exists a cubic form $G$ in $\fD$ such that the dual generator $F$  of $I$ is of the form 

\begin{enumerate}[\rm\quad(1)]
\item \label{tangentcase} ${F=X_1X_n^{\,[2]} + G(X_2,\ldots,X_n)}$, in which case, ${x_1^{\,2},x_1x_2,\ldots,x_1x_{n-1}\in I}$;
\item \label{CS13case} ${F=X_1^{\,[3]} + G(X_2,\ldots,X_n)}$, in which case, ${x_1x_2,\ldots,x_1x_n\in I}$.
\end{enumerate}
\end{proposition}

\begin{proof}  
Making a linear change of variables, we may assume that the class of $x_1$ is a linear form of rank one in $Q/I$. Then  $V\colon=(I:x_1)_1$ is  a vector subspace of $Q_1$ of dimension $n-1$.

If $x_1\in V$, then $x_1^{\,2}\in I$. After a linear change of variables we can assume that ${V=\langle x_1,\ldots,x_{n-1}\rangle}$, thus  ${x_1^2, x_1x_2,\dots,x_1x_{n-1}}$ are in $I$. The ideal $I$ is as in case \eqref{tangentcase}. It follows that the dual generator $F$ of $I$ cannot have any monomial divisible by  $X_1^{[2]}$, nor can it have a monomial divisible by  $X_1X_j$, for $2\leq j\leq n-1$. So its only monomial divisible by $X_1$ is $X_1X_n^{[2]}$, up to scaling. This monomial must occur in $F$, otherwise we would have ${x_1\in I}$, contradicting the hypothesis ${I\subseteq(x_1\dots,x_n)^2}$.  By scaling $F$, or even $X_1$, we can make the coefficient of $X_1X_n^{[2]}$ equal to $1$. This makes $F$ as in case \eqref{tangentcase}.

If $x_1\notin V$, then ${x_1^2\not\in I}$. After a change of variables we can assume that ${V=\langle x_2,\ldots,x_n\rangle}$, thus ${x_1x_2,\dots,x_1x_n}$ are in $I$. The ideal $I$ is as in case \eqref{CS13case}. 
It follows that the dual generator $F$ of $I$  does not have any monomial involving $X_1$ and any of the other variables. Since, as before, $F$ must have at least one monomial involving $X_1$, we obtain that $X_1^{[3]}$ must occur in $F$ and by scaling $F$ we can write it as in case \eqref{CS13case}.
\end{proof}

The following result is proved in \cite[Proposition 6.11, Theorem 6.15]{CRV01} for  $n\le4$ and in Caviglia \cite{Cav00} for $n=5$. Recall that a standard graded $\kk$-algebra $R$ is called Koszul if $\kk$ admits a linear resolution as an $R$-module.

\chunk
\label{conca} 
Let $\kk$ be an algebraically closed  field of characteristic zero, $Q=\kk[x_1,\dots,x_n]$ a polynomial ring, and $I$ a homogeneous Gorenstein ideal in $Q$ such that $(x_1,\dots,x_n)^4\subseteq I \subseteq (x_1,\dots,x_n)^2$.
If $n\leq 5$, the following assertions are equivalent:
\begin{enumerate}[(i)]
\item The ideal $I$ is quadratic;
\item The ring $Q/I$ has all non-zero linear forms of rank at least two;
\item The ring $Q/I$ is Kozsul.
\end{enumerate}

When $\kk$ is an algebraically closed field with characteristic zero, Proposition \ref{rankonelinearformdualgenerator}  describes the dual generators for non-quadratic ideals $I$. This characterization will play an important role in describing all non-quadratic ideals $I$, up to a linear change of variables, in the case of embedding dimension four; see Proposition \ref{mu7-9 generators}.

\section{Structure of Gorenstein ideals with 6 generators}
\label{Structure6gens}

In this section,  $\kk$ is a field and $I$ is a quadratic Gorenstein ideal in the polynomial ring $Q=\kk[x,y,z,w]$ such that $(x,y,z,w)^4\subseteq I \subseteq (x,y,z,w)^2$. The standard-graded ring $Q/I$ is thus an Artinian Gorenstein ring of embedding dimension four, socle degree three, and Hilbert function $H_{Q/I}(t)=1+4t+4t^2+t^3$. Moreover, the ideal $I$ is minimally generated by six polynomials. 
In case $\kk$ is an infinite field, by choosing a generic form $F$ of degree three in $\fD=\kk_{DP}[X,Y,Z,W]$ one obtains a quadratic ideal $I=\Ann_Q F$, as the number of generators of an ideal is a semi-continuous function and the set of cubics yielding an ideal of $6$ generators is open dense.

In order to understand the structure of such an ideal $I$ we need to recall the following definition.

\chunk 
\label{GH5}
The ring $Q/I$  is of class $\clGH(5)$, or class $\D$, as originally defined in \cite[Theorem 2.2]{KM85}, if the graded commutative  Tor algebra 
\[
\A\colon=\Tor{}{Q}{Q/I}{\kk}=\A_0\oplus \A_1\oplus \A_2\oplus \A_3\oplus \A_4,
\] 
admits bases 
\begin{equation*}
\{\ee_1,\dots, \ee_6\}\ \text{for}\  \A_1,
\ \{\ff_1,\dots,\ff_5, \ff_1', \dots, \ff_5'\}\ \text{for}\  \A_2,
\ \{\eg_1,\dots, \eg_6\}\ \text{for}\  \A_3,
\ \text{and}\ \{\hh\}\ \text{for}\  \A_4
\end{equation*}
such that for all $1\leq i\leq 5$ and  $1\leq j\leq 6$ the following equalities hold:
\begin{align*}
  \ee_6\cdot \ee_i&= \ff_i,
& \ee_i\cdot \ff_i'&=\eg_6, 
& \ee_6\cdot \ff_i'&=-\eg_i,
& \ee_j\cdot \eg_j&=\hh, 
& \ff_i\cdot \ff_i'&=\hh;
    \end{align*}
all other products are zero. In particular, $\rank_k \A_1\cdot \A_1=5$. 

\begin{theorem}
\label{Iquadratic}
Let $\mathsf{k}$ be an algebraically closed field with $\cha\kk=0$ and $I$ a homogeneous grade four quadratic Gorenstein ideal of the polynomial ring $Q=\kk[x,y,z,w]$ containing the ideal $(x,y,z,w)^4$.
Then, the following conditions hold:
\begin{enumerate}[$(a)$]
    \item \label{IquadraticTor} The ring $Q/I$ is of class $\clGH(5).$
    \item \label{IquadraticJ} There exist a quadratic form ${f}$ and a grade three Gorenstein ideal ${J}$ of $Q$, minimally generated by five quadrics, such that:
    \[
    {J:(f)=J}\quad \text{and} \quad  {J+(f)=I}.
    \]\item There exists a quadratic form $f$ and 
    a grade three Gorenstein ideal ${J}$ of $Q$ minimally generated by five quadrics such that the following sequence is exact:
    \[
    0\to (Q/J)(-2)\xra{f\cdot} Q/J\xra{\pi} Q/I\to 0,
    \]
    where $\pi$ is the canonical projection.
\end{enumerate}
\end{theorem}

\begin{proof} \eqref{IquadraticTor}:  Since the ring $Q/I$ is quadratic, it is  Koszul by \ref{conca}.  Therefore,  we have the first equality of the following:
\[\Po{Q/I}{\kk}(t)=\frac{1}{\HF{Q/I}(-t)}=\frac{1}{1-4t+4t^2-t^3}=\frac{(1+t)^4}{(1+t)^2(1-2t-3t^2+3t^3+2t^4-t^5)}.
\]
The second equality uses the hypothesis and the third is obtained by multiplying the fraction by $\frac{(1+t)^4}{(1+t)^2(1+t)^2}.$ 
By the classification of codepth four Gorenstein rings in \cite[Table 1, page 50]{Avr89}, based on the  structure of the Tor algebra $\A$, we conclude that the ring $Q/I$ is of class $\clGH(5)$; see \ref{GH5}. 

\eqref{IquadraticJ}: By part \eqref{IquadraticTor}, the ring $Q/I$  is of class $\clGH(5)$. If we consider the bases of $\A_i$ for $1\leq i\leq 4$, as in \ref{GH5},
it is easy to check that  $\A$ is a free algebra over its sub-algebra  $\kk[\ee_6]=\kk\oplus \kk\cdot\ee_6$, with basis $\{1, \ee_1,\dots, \ee_5, \ff_1',\dots, \ff_5', \eg_6\}$. 
Consider $(\bF,\partial)$ a minimal free graded resolution of the $Q$-module $Q/I$
such that the standard bases for $F_i$ are the liftings of the bases above for $\A_i$. If $\{e_1, \dots, e_6\}$ denotes the  standard basis for $F_1 = Q^6$, then set 
\[
f\colon= \partial_1(e_6).
\]
The quadratic form $f$ is regular in $Q$, as $f\not=0$ and $Q$ is a domain. By the proof of \cite[Lemma 3.3]{Avr89}, we have 
\begin{equation*}
\A/\ee_6\A\cong\Tor{}{Q/(f)}{Q/I}{\kk}\cong\Tor{}{Q/(f)}{{Q/(f)}/{I/(f)}}{\kk}.
\end{equation*}
 Since $Q/I$ is a Gorenstein ring, the algebra $\A$ has Poincar\'e duality \cite{AG71} (see also \cite[Theorem 3.4.5]{BH93}). 
 Using that  $\A$ is a free algebra over $\kk[\ee_6]$ we obtain that the algebra $\A/\ee_6\A$ also has Poinar\'e duality.
Hence, again by \cite{AG71}, it follows that  $I/(f)$ is a grade three Gorenstein ideal of $Q/(f)$.
By the Buchsbaum-Eisenbud Structure Theorem of grade three Gorenstein ideals, see \cite[Theorem 2.1(2)]{BE77a} (see also \cite[Theorem 3.4.1(b)]{BH93}), there exists a skew\h{symmetric} $5\times 5$ matrix $\widetilde T$ with entries in $Q/(f)$ such that the ideal $I/(f)$ is given by the sub-maximal Pfaffians of the matrix $\widetilde T$.
Consider a skew-symmetric $5\times 5$ matrix $T$ with coefficients in $Q$ such that $T\otimes Q/(f)=\widetilde T$ and define $J$ to be the ideal generated by the sub-maximal Pfaffians of the matrix $T$. In particular, we have $I=(f)+J$ and so $J$ is minimally generated by five quadrics. Since $\grade{Q}{I}=4$ and $f$ is regular,  we have $\grade{Q}{J}=3$. Thus, again by the Buchsbaum-Eisenbud Structure Theorem, see \cite[Theorem 2.1(1)]{BE77a} (see also \cite[Theorem 3.4.1(a)]{BH93}),  we  conclude that $J$ is a Gorenstein ideal of $Q$. Set $\overline f:=f+J$. The equality $J:(f)=J$ holds if and only if $\overline f$ is a regular element of $S=Q/J$. This follows from the fact that the ring  $Q/I\cong S/(\overline f)$ is Artinian  and $S$ has depth one, as $J$ is a grade three perfect ideal in the regular ring $Q$. 

(c): If follows directly from part (b). \end{proof}

\begin{corollary}
\label{cor: Iquadratic}
Let $\mathsf{k}$ be an algebraically closed field with $\cha\kk=0$ and $I$ a homogeneous grade four quadratic Gorenstein ideal of the polynomial ring $Q=\kk[x,y,z,w]$ containing the ideal $(x,y,z,w)^4$. Then, the following assertions hold.
\begin{enumerate}[$(a)$]
\item There exists a quadratic Gorenstein grade three ideal $J$ and a quadratic form $f$ yielding an embedding $\iota$ such that the  short sequence
\[
0\to \omega_{Q/J}(-3)\xra{\iota} Q/J\xra{\pi} Q/I\to 0,
\]
where $\pi$ is the canonical projection, is exact.
\item A graded minimal free resolution of $Q/I$ over $Q$ has the form:
\[
\bF\colon 0\to 
Q(-7)\xra{}
Q^6(-5)\xra{}
\begin{matrix}Q^5(-4)\\ \oplus\\ Q^5(-3)\end{matrix}
\xra{}
Q^6(-2)\xra{}
Q.
\]
\end{enumerate}
\end{corollary}

\begin{proof} (a):  Let $J$ be the quadratic grade three Gorenstein ideal and $f$ the qudratic form  given by  Theorem \ref{Iquadratic}. By the Buchsbaum-Eisenbud Structure Theorem \cite[Theorem 2.1]{BE77a} (see also \cite[Theorem 3.4.1(b)]{BH93}), there exists a minimal  free resolution  of $Q/J$ over $Q$  of the following form:
\[\bG\colon
0\to
Q\xra{(\partial_1^\bG)^*}
Q^5\xra{\partial_2^\bG}
Q^5\xra{\partial_1^\bG}
Q,
\]
 where $\partial_2^\bG=-(\partial_2^\bG)^*$.
 Moreover, $J$ is generated by the  sub\h{maximal} Pfaffians of  the $5\times 5$  skew\h{symmetric} matrix $\partial_2^\bG$. Since $J$ is generated by quadrics, $\partial_2^\bG$ has linear entries.
 Therefore, a graded minimal free resolution  of $Q/J$ over $Q$  has the following form:
\[\bG\colon
0\to
Q(-5)\xra{(\partial_1^\bG)^*}
Q^5(-3)\xra{\partial_2^\bG}
Q^5(-2)\xra{\partial_1^\bG}
Q,
\]
By \eqref{Gor canonical}, the canonical module of $Q/J$ is $(Q/J)(-4+5)=(Q/J)(1)$, thus in the short exact sequence from Theorem \ref{Iquadratic}(c) one may replace $(Q/J)(-2)$ by $\omega_{Q/J}(-3)$, and $\iota\colon \omega_{Q/J}(-3)\to Q/J$ is given by multiplication by $f$. 

(b): The multiplication map $(Q/J)(-2)\xra{f\cdot} Q/J$ extends to a
chain map of free resolutions:
\begin{equation*}
 \xymatrixrowsep{1pc}
 \xymatrixcolsep{1pc}
 \xymatrix{
 \bG(-2)\colon\ar@{->}[d]^{f\cdot}&0\ar@{->}[r]
 &Q(-7)\ar@{->}[r]\ar@{->}[d]^{f\cdot}
 &Q^5(-5)\ar@{->}[r]\ar@{->}[d]^{f\cdot}
 &Q^5(-4)\ar@{->}[r]\ar@{->}[d]^{f\cdot}
 &Q(-2)\ar@{->}[d]^{f\cdot}
 \\
 \bG\colon&0\ar@{->}[r]
 &Q(-5)\ar@{->}[r]
 &Q^5(-3)\ar@{->}[r]
 &Q^5(-2)\ar@{->}[r] 
 &Q.
 }
\end{equation*}
The isomorphism $\Sigma^{-3}\bG^*\cong \bG$ and the embedding $\iota$ induce a homogeneous chain map $\overline\iota\colon \Sigma^{-3}\bG^*(-2)\to \bG$. By Lemma \ref{exact F}, the mapping cone of $\overline\iota$, denoted by $\bF,$ is a graded  free resolution of the $Q$\h{module} $Q/I$ of the desired form. Since all the differential maps $\partial^\bF_i$ have entries in $(x,y,z,w)$, the resolution $\bF$ is also minimal.
\end{proof}

\section{Non-quadratic Gorenstein ideals}
\label{NonquadGorideals}

In this section,  $\kk$ is an algebraically closed field with $\cha\kk=0$ and $I$ is a homogeneous non-quadratic Gorenstein ideal in the polynomial ring $Q=\kk[x,y,z,w]$ such that $(x,y,z,w)^4\subseteq I \subseteq (x,y,z,w)^2$. The standard-graded ring $Q/I$ is thus an Artinian Gorenstein ring of embedding dimension four, socle degree three, and the Hilbert function $\HF{Q/I}{(t)}=1+4t+4t^2+t^3$. Moreover, the ideal $I$ is minimally generated by at least seven polynomials. 
 
The following statement is a particular case of Proposition \ref{rankonelinearformdualgenerator}.  It  will allow us to give a precise description of all non-quadratic ideals $I$ as above.

\chunk
\label{nonquadraticdual}
The ring $Q/I$ admits a linear form of rank one if and only if after a linear change of variables, there exists a form $G$ of degree three in the divided power ring $\fD=\kk_{DP}[X,Y,Z,W]$ such that the dual generator $F$ of $I$ is of the form 
\begin{enumerate}[$(1)$]
\item\label{nonquadraticdualsquare} ${F=XY^{[2]} + G(Y,Z,W)}$, or 
\item\label{nonquadraticduaconnectedsum} ${F=X^{[3]} + G(Y,Z,W)}$.
\end{enumerate}

\chunk
\label{connected sum} Let $F\in\fD_3$ be a dual generator of the ideal $I$.
By  Ananthnarayan, Avramov, and Moore \cite{AAM12} the ring  $Q/I$ is a \emph{connected sum} if, after a change of variables, there exist $F'$ and $F''$ in $\fD_3$ such that the dual generator $F$ can be written as
\[
F = F'(X) + F''(Y,Z,W)\qquad\text{or}\qquad F = F'(X,Y) + F''(Z,W).
\]
In the first case, we say that the connected sum is of \emph{type} $(1,3)$ and in the second, it is of \emph{type} $(2,2)$.

\begin{remark}
\label{13-22}
A connected sum can be of both types $(1,3)$ and $(2,2)$, e.g.\ the Artinian Gorenstein algebra whose dual generator is the Fermat cubic is ${F=X^{[3]}+Y^{[3]}+Z^{[3]}+W^{[3]}}$. Note that by \cite[Proposition 4.1]{ACLY19}, the ring  $Q/I$ is a connected sum of type $(1,3)$ if, after a change of variables, the ideal $I$ contains $xy$, $xz$, and $xw$; it is of type $(2,2)$ if, after a change of variables, the ideal $I$ contains $xz$, $xw$, $yz$, and $yw$.
\end{remark}

\begin{example}
\label{excs22}
In the proof of the next result Proposition \ref{mu7-9 generators}, we will come across the dual generator
\[
F = XY^{[2]} +  Y(Z^{[2]}+ZW+aW^{[2]}) + Z^{[3]} + W^{[3]},
\]
with ${a\in\kk}$. The annihilator of $F$ is the ideal
\[
I=(x^2,\, xz,\, xw,\, xy+yz-z^2,\, a^2xy+yw-aw^2,\, zw,\, xy^2-w^3,\, y^3,\, z^3-w^3).
\]
We show that ${Q/I}$ is a connected sum: Consider the new variables
\begin{align*}
    x'&=x, & y'&=y-z-aw; & z'&=z+x; & w'=w+a^2x. 
\end{align*}
Then
\begin{align*}
    x'z'&=x(z+x)=xz+x^2;\\
    x'w'&=x(w+a^2x)=xw+a^2x^2;\\
    y'z'&=(y-z-aw)(z+x)=xy+yz-z^2-xz-axw-azw;\\ 
    y'w'&=(y-z-aw)(w+a^2x)=a^2xy+yw-aw^2-a^2xz-a^3xw-zw. 
\end{align*}
Therefore ${x'z',x'w',y'z',y'w'\in I}$, so  by Remark \ref{13-22} we conclude that $Q/I$ is a connected sum of type $(2,2)$.
\end{example}

The next result is a classification of non-quadratic Gorenstein ideals $I$ dicussed in this section.  Our proof uses some techniques from Casnati and Notari \cite{CN11}.
\begin{proposition}
\label{mu7-9 generators}
Let $\mathsf{k}$ be an algebraically closed field with $\cha\kk=0$ and $I$ a non-quadratic homogeneous ideal of the polynomial ring ${Q=\mathsf{k}[x,y,z,w]}$ such that $(x,y,z,w)^4 \subseteq I\subseteq (x,y,z,w)^2$.  Then, the ideal $I$ is minimally generated by seven or nine elements, and  after a change of variables, there exist $a,b\in \kk$ such that $F$ and $I$ admit one of the following descriptions:

 \begin{enumerate}[I.]
\item \textbf {Seven generators case}
\label{mu7-9seven}
\begin{flalign*}
\hspace{0.5cm}
{\mathbf a.}\quad  
  F&=XY^{[2]}  + Y(aZ^{[2]}+ZW+bW^{[2]})+Z^{[3]} + W^{[3]}&\\
   I&=\bigl(x^2,\, xz,\, xw,\, xy-zw,\, yz-az^2+(a^2+b)zw-w^2& \\
   &\hspace{3.9cm}yw-z^2+(a+b^2)zw-bw^2,\, y^3\bigr),&\\
   &&\\
{\mathbf b.}\quad 
   F&=XY^{[2]} + Y(Z^{[2]}+aZW)+ZW^{[2]}&\\
   I&=\bigl(x^2,\, xz,\, xw,\, xy-z^2,\, yz+a^2z^2-azw-w^2,\,
      yw-aw^2,\, y^3\bigr).&
\end{flalign*}

\item \label{mu7-9connectedsum7}
\textbf {Seven generators case}
\begin{flalign*}
\hspace{0.5cm}
{\mathbf a.}\quad 
    F&=X^{[3]} + Y^{[2]}Z - aZ^{[2]}W + (a+1)ZW^{[2]} - 3W^{[3]} &\\
    I&=\bigl(xy,\, xz,\, xw,\, yw,\, a^2y^2+azw+(a+1)z^2,&\\
     &\hspace{2cm}  z^2+a(a+1)bzw+a^2bw^2,\, x^3-y^2z\bigr),&\\
     &\qquad\text{with}\ a(a-1)(a^2-a+1)\not=0\quad\text{and}\quad b=(a^2-a+1)^{-1}.&\\
     &&\\
{\mathbf b.}\quad   
    F&=X^{[3]} + YZW - Z^{[3]} + W^{[3]}&\\
    I&=(xy,\, xz,\, xw,\, y^2,\, yz-w^2,\, yw+z^2,\, x^3-w^3),&\\
     &&\\
{\mathbf c.}\quad  
    F&=X^{[3]} + Y^{[3]} + YZ^{[2]} - YW^{[2]}&\\
    I&=(xy,\, xz,\, xw,\, y^2+w^2,\, z^2+w^2,\, zw,\, x^3-yz^2),&\\
     &&\\
{\mathbf d.}\quad 
    F&=X^{[3]} + YZW&\\
    I&=(xy,\, xz,\, xw,\, y^2,\, z^2,\, w^2,\, x^3-yzw).&
   \end{flalign*}
\item 
\label{mu7-9nine}
\textbf{Nine  generators case}
\begin{flalign*}
\hspace{0.5cm}
{\mathbf a.}\quad  
    F&=XY^{[2]} + YZW + ZW^{[2]}&\\
   I&=(x^2,\, xz,\, xw,\, xy+yz-zw,\, yw-w^2,\, z^2, \, y^2z,\, y^3,\, w^3),&\\
     &&\\
{\mathbf b.}\quad  
    F&=XY^{[2]} + Y(Z^{[2]}+aZW+bW^{[2]})+W^{[3]}&\\
    I&=\bigl(x^2,\, xz,\, xw,\, xy-z^2,\, az^2-zw,&\\
     &\hspace{1cm} ayz-yw+b(a^2-b)z^2+(b-a^2)w^2,\,
    xy^2-w^3,\, y^3,\, y^2z\bigr),&\\
     &&\\
{\mathbf c.}\quad  
    F&=XY^{[2]} + Y(ZW+aW^{[2]})+W^{[3]}&\\
    I&=(x^2,\, xz,\, xw,\, xy-zw,\, z^2,\, 
    axy+yz-w^2, y^3,\, y^2w,\, yw^2-aw^3),&\\
     &&\\
{\mathbf d.}\quad  
    F&=XY^{[2]}+ZW^{[2]}&\\
    I&=(x^2,xz,xw,yz,yw,z^2,y^3,xy^2-zw^2,w^3).&
\end{flalign*}
\item \textbf{Nine  generators case}
\label{mu7-9connectedsum9}
\begin{flalign*}
\hspace{0.5cm}
{\mathbf a.}\quad  
    F&= X^{[3]} + Y^{[2]}Z - aZ^{[2]}W + (a+1)ZW^{[2]} - 3W^{[3]}&\\
    I&=\bigl(xy,\, xz,\, xw,\, yw,\, y^2-(a-1)zw-(2a-1)z^2,\,&\\
    &\hspace{2cm} (2a-1)zw+(a-1)w^2,\, 3x^3+w^3,\, y^3,\, z^3 \bigr),&\\
    &\qquad\text{where}\ a^2-a+1=0,&\\
    &&\\
{\mathbf b.}\quad  
    F&=X^{[3]} + YZ^{[2]} + W^{[3]}\\
    I&=(xy,\, xz,\, xw,\, y^2,\, yw,\, zw, \, x^3-w^3,\, yz^2 - w^3,\, z^3),&\\
     &&\\
    {\mathbf c.}\quad  F&=X^{[3]} + Y^{[2]}Z + YW^{[2]}&\\
    I&=(xy,\, xz,\, xw,\, yz-w^2,\, z^2,\, zw,\, x^3-y^2z,\, y^3,\, y^2w).&
\end{flalign*}
\end{enumerate}
\end{proposition}

\begin{proof} 
We follow a characterization from the proof of \cite[Proposition 4.6]{CN11}, with a slight adaption, mostly because we are working with contraction rather than differentiation. By \ref{conca}, the ring ${Q/I}$ admits a linear form of rank one as $I$ is non-quadratic ideal. Moreover, since the Hilbert function of $Q/I$ is $\HF{Q/I}{(t)}=1+4t+4t^2+t^3$, the ideal $I$ is minimally generated by exactly  six quadrics and at least one more generator. We prove that besides the six quadrics the ideal $I$ is minimally generated by either one or three cubics, depending whether the dimension of the $\kk$-vector space $Q_1\cdot I_2$ is $18$ or $16$, respectively. 
By \ref{nonquadraticdual}, $Q/I$ admits a dual generator of one of the form (\ref{nonquadraticdualsquare}) or (\ref{nonquadraticduaconnectedsum}). Once we have the description of the dual generator, the corresponding list of minimal generators of $I$ can be obtained by straightforward (even if tedious) computation.  

First, consider the case (\ref{nonquadraticdualsquare}) of  \ref{nonquadraticdual}, when after a change of variables we assume ${F=XY^{[2]} + G(Y,Z,W)}$. There exist a linear form $L\in\fD_1$, a cubic $C\in\fD_3$, and $e_1,e_2,e_3\in\kk$ such that:
\[
G(Y,Z,W)= Y^{[2]}L(Y,Z,W) + Y(e_1Z^{[2]}+e_2ZW+e_3W^{[2]}) + C(Z,W),
\]
hence $F=\big(X+L(Y,Z,W)\big)Y^{[2]}+Y(e_1Z^{[2]}+e_2ZW+e_3W^{[2]}) + C(Z,W)$. By replacing ${X+L(Y,Z,W)}$ with $X$, we can further assume that 
\begin{equation}
\label{mu7-9 generators_dualF}
F=XY^{[2]}+Y(e_1Z^{[2]}+e_2ZW+e_3W^{[2]}) + C(Z,W).
\end{equation}
Since $\kk$ is algebraically closed, $C$ has either three simple roots, or one double root and a simple root, or a triple root. 

If $C$ has \emph{three simple roots}, with a linear change of variables, we can assume ${C=Z^{[3]}+W^{[3]}}$, as $e_1$, $e_2$ and $e_3$ will change accordingly, therefore $F$ becomes:
\[
F = XY^{[2]} +  Y(e_1Z^{[2]}+e_2ZW+e_3W^{[2]}) + Z^{[3]} + W^{[3]}.
\]
If ${e_2\ne0}$, by replacing $Y$ by $e_2^{\,-1}Y$ and scaling  $X$, we get the case \ref{mu7-9seven}a, where $e_1$ is replaced by $a$ and $e_3$ by $b$. If ${e_2=0}$, then we can again scale $Y$ and $X$, and get the case we explored in Example \ref{excs22} (making ${e_3=a}$). Since it is a connected sum of type $(2,2)$, we may write ${F = F'(X,Y) + F''(Z,W)}$. Being a cubic in two variables, $F'$ can be written as ${F'=X^{[3]}+Y^{[3]}}$, if it has three simple roots, or ${F'=XY^{[2]}}$, if it has a double root; it cannot have a triple root, or we would have a linear element in $I$, contradicting the hypothesis that ${I\subseteq(x,y,z,w)^2}$.
If ${F'=X^{[3]}+Y^{[3]}}$, we are in case \eqref{nonquadraticduaconnectedsum} of \ref{nonquadraticdual}. We will address this case later in this proof. So we are left with ${F'=XY^{[2]}}$ and likewise ${F''=ZW^{[2]}}$, bringing us to case \ref{mu7-9nine}d.

If $C$ has \emph{one double root and a simple root}, again by a linear change of variables, we can assume ${C=ZW^{[2]}}$, and by replacing $Z$ with ${-e_3Y+Z}$ in \eqref{mu7-9 generators_dualF}, we get:
\begin{align*}
F &= XY^{[2]} + Y\bigl(e_1(-e_3Y+Z)^{[2]}+e_2(-e_3Y+Z)W+e_3W^{[2]}\bigr) + (-e_3Y+Z)W^{[2]}\\
&= XY^{[2]} + Y(e_1e_3^{\,2}Y^{[2]}-e_1e_3YZ+e_1Z^{[2]}-e_2e_3YW+e_2ZW) + ZW^{[2]}\\
&= (X+e_1e_3^{\,2}Y-2e_1e_3Z-2e_2e_3W)Y^{[2]} + Y(e_1Z^{[2]}+e_2ZW) + ZW^{[2]}.
\end{align*}
\\
If we replace the linear expression ${X+e_1e_3^{\,2}Y-2e_1e_3Z-2e_2e_3W}$ by $X$, we get
\[
F= XY^{[2]} + Y(e_1Z^{[2]}+e_2ZW) + ZW^{[2]}.
\]
If ${e_1\ne0}$, by scaling $Y$ and $X$ we get to case \ref{mu7-9seven}b. If ${e_1=0}$ but ${e_2\ne0}$, scaling again $Y$ and $X$, we get to case \ref{mu7-9nine}a. If ${e_1=e_2=0}$, we get again to case \ref{mu7-9nine}d.

Finally, if $C$ has \emph{one triple root}, we can assume ${C=W^{[3]}}$, thus $F$ from \eqref{mu7-9 generators_dualF} becomes:
\[
F = XY^{[2]} +  Y(e_1Z^{[2]}+e_2ZW+e_3W^{[2]}) + W^{[3]}.
\]
Note that $e_1$ and $e_2$ cannot both be zero, since in that case, we would have ${z\in I}$, contradicting the hypothesis that ${I\subseteq(x,y,z,w)^2}$. If ${e_1\ne0}$, we get case \ref{mu7-9nine}b. If ${e_1=0}$, then we must have ${e_2\ne0}$, and we get case \ref{mu7-9nine}c.

Second, consider the case \eqref{nonquadraticduaconnectedsum} of \ref{nonquadraticdual}, when after a change of variables we may assume ${F=X^{[3]} + G(Y,Z,W)}$. Next, we consider the  classification of the ternary cubics from the book of Harris \cite[Example 10.16]{Har92}.

If $G$ is a \emph{non-singular cubic}, then by a change of variables, we can write 
\[
G = Y^2Z - W(W-Z)(W-aZ) = Y^2Z -aZ^2W + (1+a)ZW^2 - W^3
\] 
with ${a\notin\{0,1\}}$ (note that we are using usual powers here). Moreover, two such cubics are projectively equivalent if and only if their $j$-invariants
\[
j(a)=256\cdot\tfrac{(a^2-a+1)^3}{a^2(a-1)^2}
\]
are the same. For a brief description of the $j$-invariants, see \cite[Example 10.12, Exercise 10.13]{Har92}.

Using that the relation between divided powers and usual ones is given by ${X^n=n!X^{[n]}}$, we get 
$F = X^{[3]} + 2Y^{[2]}Z - 2aZ^{[2]}W + 2(a+1)ZW^{[2]} - 6W^{[3]}.$
Dividing by $2$ and scaling $X$, we get 
\[
F=X^{[3]} + Y^{[2]}Z - aZ^{[2]}W + (a+1)ZW^{[2]} - 3W^{[3]}
\]
If ${a^2-a+1\ne0}$, we get to case \ref{mu7-9connectedsum7}a, where $I$ admits seven generators; otherwise we can note that both roots of ${a^2-a+1}$ yield a zero $j$-invariant, so we can choose one of them, and we get to case \ref{mu7-9connectedsum9}a.

If $G$ is \emph{an irreducible cubic with a node singularity}, then after a change of variables we can write ${G=YZW + Z^{[3]} + W^{[3]}}$,  and replacing $Z$ by $-Z$ and $Y$ by $-Y$, we get to case \ref{mu7-9connectedsum7}b.

If $G$ is \emph{an irreducible cubic with a cusp singularity}, then after a change of variables we can assume that ${G=YZ^{[2]} + W^{[3]}}$, and we get to  case \ref{mu7-9connectedsum9}b.

If $G$ is \emph{the union of a conic and a secant line}, then after a change of variables we can assume that  ${G=Y^{[3]} + YZ^{[2]} - YW^{[2]}}$, and we get to case \ref{mu7-9connectedsum7}c.

If $G$ is \emph{the union of a conic and a tangent line}, then after a change of variables  we can assume that 
${G=Y(YZ + W^{[2]})=2Y^{[2]}Z + YW^{[2]}}$
and by scaling $Z$ we get case \ref{mu7-9connectedsum9}c.

If $G$ is \emph{ union of three lines not meeting at a point}, then after a change of variables we can assume ${G=YZW}$, and we get to case \ref{mu7-9connectedsum7}d.

All other cases for $G$ imply that $I$ would have a linear form, contradicting the hypothesis.
\end{proof}

\begin{example} 
Consider the form  $F=X^{[3]}+XYZ+XYW+XZW+YZW$ in $\fD=\kk_{DP}[X,Y,Z,W]$. Then $\Ann_QF$ in $Q=\kk[x,y,z,w]$ is a quadratic ideal given by:
\[
\Ann_QF=(2x^2+xw-yw-zw,\, xy-yz-xw+zw,\, xz-yz-xw+yw,\,  y^2,\, z^2,\, w^2).
\]
This is the case because $\kk$ has characeristic zero. However, if we consider $F$ over a field of characteristic two, its annihilator ideal is not quadratic (the coefficient of $x^2$ in the previous set of generators vanishes). If $\kk'$ is the algebraic closure of $\zz_2$ and $Q'=\kk'[x,y,z,w]$, then  $\Ann_{Q'}F$ is not a quadratic ideal, as it is given by: 
\[
\Ann_{Q'}F=(xy+yz+yw,\, xz+yz+zw,\, xw+yw+zw,\,  y^2,\, z^2,\, w^2,\, x^3+yzw).
\]
This example shows that the hypothesis that $\kk$ has characteristic zero is crucial for Proposition \ref{mu7-9 generators}. A characteristic\h{free} result would have a much longer list of cases.
\end{example}

\section{Structure of Gorenstein ideals with 7 generators}
\label{Structure7gens}

The main result of this section is the following.

\begin{theorem}
\label{mu7 structure}
Let $\mathsf{k}$ be an algebraically closed field with $\cha\kk=0$ and $I$ a homogeneous Gorenstein ideal of the ring  $Q=\kk[x,y,z,w]$, minimally generated by seven elements, with $(x,y,z,w)^4\subseteq I\subseteq(x,y,z,w)^2$.
Then, there exists a grade three homogeneous perfect ideal $J$ of resolution format $\ff_J=(1,5,6,2)$, such that  $J\subset I$, and a
homogeneous embedding $\iota$  of $\omega_{Q/J}(-3)$
into $Q/J$ such that the following sequence is exact:
\[
0\rightarrow \omega_{Q/J}(-3)\xra{\iota} Q/J\xra{\pi} Q/I\rightarrow 0,
\]
where $\pi$ is the canonical projection.
\end{theorem}

The proof of this theorem, given at the end of this section,  uses the description given in Proposition \ref{mu7-9 generators} of all possible ideals $I$ that are  minimally generated by seven elements, up to a linear change of variables, and Propositions \ref{J res mu7} and \ref{J res mu7-c}.

\begin{setup}
\label{mu7-setup} We consider the ideals $I$ as in  Proposition \ref{mu7-9 generators}.I.  
First, for each one of the cases {Ia and Ib}, we identify an ideal $J$ of $Q$ such that $J\subset I$.
Second, we use a terminology due to Brown \cite{Bro87} to  give a uniform description of the ideals $J$ that satisfy the conditions from Theorem \ref{mu7 structure}; see also Proposition \ref{J res mu7}.

\begin{enumerate}[\bf a.]
\item[\bf Ia.] Let $a,b\in \kk$. We may write  
\begin{flalign*}
\hspace{0.5cm}
I&=J+(xy-zw,y^3),\ \text{where}&\\
J&=(x^2,\,xz,\,xw,\,yz-az^2+(a^2+b)zw-w^2,\,yw-z^2+(a+b^2)zw-bw^2).&
\end{flalign*}

\item[\bf Ib.] Let $a\in \kk$. We may write   
\begin{flalign*}
\hspace{0.5cm}
I&=J+ \bigl(\, xy-z^2,\, y^3\bigr),\ \text{where}&\\  
J&=(x^2,\,xz,\, xw,\,yz+a^2z^2-azw-w^2,\,yw-aw^2).&
\end{flalign*}
\end{enumerate}
 If we consider 
\begin{equation*}
  \arraycolsep=5pt
  \def\arraystretch{1.5}
  \begin{array}{c|cccc}
    \text{\scriptsize Case } &A&B&C&D\\
    \hline
    \text{\scriptsize \bf Ia}&y-bw&z-(a+b^2)w&(a^2+b)z-w&-y+az\\
    \text{\scriptsize \bf Ib}&y-aw&0&-az-w&-y-a^2z\\
  \end{array}
\end{equation*}
then the non-zero submaximal Pfaffians of the skew-symmetric matrix
\begin{equation*}
  \label{matrix}
T=\begin{pmatrix}
0&0&A&B&0\\
0&0&C&D&0\\
-A&-C&0&x&z\\
-B&-D&-x&0&w\\
0&0&-z&-w&0
\end{pmatrix}
\end{equation*}
obtained by removing rows and columns 1, 2, and 5, are  given respectively by:
 
\begin{align*}
    T_1&=wC-zD=\begin{cases}
    yz-az^2+(a^2+b)zw-w^2,&\text{if\ {Ia}}\\
    yz+a^2z^2-azw-w^2,&\text{if\ {Ib}.}
    \end{cases}
    \\
    T_2&=wA-zB=
    \begin{cases}
    yw-z^2+(a+b^2)zw-bw^2&\text{if\ {Ia}}\\
    yw-aw^2,&\text{if\ {Ib}.}
    \end{cases}
    \\
    T_5&=BC-AD=
    \begin{cases}
    y^2-ayz-byw+(a^2+b)z^2&\\
    \hspace{1cm}
    -(a^3+a^2b^2+b^3+1)zw+(a+b^2)w^2&\text{if\ {Ia}}\\
    y^2+a^2yz-ayw-a^3zw,&\text{if\ {Ib}.}
    \end{cases}
 \end{align*}
Moreover, the following equalities hold:
\begin{equation*}
    J=(x^2,\ xz,\ xw,\ T_1,\  T_2)\quad\text{and}\quad
    I=J+ (q,\, y^3),
\end{equation*}
where 
\[q= xy -p\quad\text{where}\quad p= \begin{cases}zw,&
    \text{if\ {Ia}}\\
    z^2,&\text{if\ {Ib}.}
    \end{cases}
\]
It is easy to check  that the following equalities hold:
\begin{align}
\label{iden1}
zT_5&=AT_1-CT_2,\quad
wT_5=BT_1-DT_2,\quad\text{and}\\\
\label{iden2}
    T_5&=y^2+
\begin{cases}
-aT_1-bT_2-(ab - 1)^2zw,&\ 
\text{if Ia}\\
a^2T_1-aT_2-a^4z^2,&\ 
\text{if Ib}.
\end{cases} 
\end{align}
\end{setup}

We denote by $[-]$, the operation of taking modulo the ideal $J$ of an element in $Q$.
Based on a description from \cite[Theorem 4.4]{Bro87} have the following result:

\begin{proposition}
\label{J res mu7}
Adopt Setup  \ref{mu7-setup}. The following assertions hold.
\begin{enumerate}[$(a)$]
\item The ideal $J$ is a grade three perfect homogeneous  ideal with a graded minimal free resolution of $Q/J$ over $Q$ given by  
\[
{\bG}\colon 0\to 
\begin{matrix}Q(-5)\\ \oplus\\ Q(-4)\end{matrix}
\xra{\partial_3}
\begin{matrix}Q(-4)\\ \oplus\\ Q^5(-3)\end{matrix}
\xra{\partial_2}Q^5(-2)\xra{\partial_1}Q,
\]
where
   \begin{align*}
     \partial_3&=
     \begin{pmatrix}
     x&0\\
     T_1&0\\
     -T_2&0\\
     0&-w\\
     0&z\\
     T_5&-x
     \end{pmatrix},&
     \partial_2&=
     \begin{pmatrix}
       T_2&0&x&0&0&0\\
       T_1&-x&0&0&0&0\\
       0&A&C&0&-x&-z\\
       0&B&D&x&0&-w\\
       0&0&0&z&w&0
      \end{pmatrix},\\
     &\hspace{2cm}\text{and}&\partial_1&=\begin{pmatrix} T_1&-T_2&-xw&xz&-x^2\end{pmatrix}.
     \end{align*}
     \item The sequence 
$\begin{matrix}Q(-3)\\ \oplus\\ Q^5(-4)\end{matrix}
\xrightarrow{-\partial_3^*}
\begin{matrix}Q(-2)\\ \oplus\\ Q(-3)\end{matrix}
\xrightarrow{\varphi}I/J\to 0$
is exact, where 
\begin{align*}
\varphi&=
\begin{cases}
\begin{pmatrix}[xy-zw]&[y^3+(ab-1)^2y(xy-zw)]\end{pmatrix},&  \text{if}\  {Ia}\\
\begin{pmatrix}[xy-z^2]&[y^3+a^4y(xy-z^2)]\end{pmatrix},&
 \text{if}\  {Ib.}
 \end{cases}
\end{align*}
\end{enumerate}
\end{proposition}

\begin{proof}
(a): It is clear that $J=\Coker\partial_1$. By definitions of $T_1,T_2$ and by \eqref{iden1} we get $\partial_1\partial_2=0$ and    $\partial_2\partial_3=0$, respectively. 
Thus $\partial_2^*\partial_1^*=0$ and $\partial_3^*\partial_2^*=0$. Therefore, both $\bG$ and $\Sigma^{-3}\bG^*$ are complexes. Using Lemma \ref{reg seq I} it is easy to see that following grade inequalities hold:
\begin{align*}
  \grade{}{I_1(\partial_1)}&\geq 3,\quad \text{as $I_1(\partial_1)$ contains the regular sequence $\{x^2,\,  T_1,\, T_2\}$}.\\
  \grade{}{I_4(\partial_2)}&\geq 2,\quad \text{as $I_4(\partial_2)$ contains the regular sequence $\{x^4,\, zw^2T_2^2\}$, and}\\
  \grade{}{I_2(\partial_3)}&\geq 3,\quad \text{as $I_2(\partial_3)$ contains the regular sequence $\{x^2,\, wT_1,\,  zT_5\}$}.
\end{align*}
By the Acyclicity Criterion \cite[Theorem 1.4.13]{BH93} it follows that both complexes $\bG$  and $\Sigma^{-3}{\bG}^*$ are acyclic, thus $J$ is a perfect ideal of grade three.

(b):  The surjectivity of $\varphi$ is clear. Indeed, if we set \begin{align*}
    p=\begin{cases}
       zw&\ \text{in case {Ia}}\\
       z^2&\ \text{in case {Ib}}
       \end{cases}
\quad
       \text{and}
 \quad      
       c=\begin{cases}
       (ab-1)^2&\ \text{in case {Ia}}\\
       a^4&\ \text{in case {Ib}}.
       \end{cases}
\end{align*}
and we further set  $q=xy-p$, then we may write
\begin{equation*}
    I=J+(q,\, y^3)\quad\text{and}\quad\varphi=\big([q]\quad [y^3+cyq]\big).
\end{equation*}
Remark that by \eqref{iden2} we have $[T_5]=[y^2-cp]$.

The inclusion $\im \partial_3^*\subseteq\Ker\varphi$ is equivalent to the equality  $\varphi\partial_3^*=0$. 
Remark that 
\[\partial_3^*=
\begin{pmatrix}
x&T_1&-T_2&0&0&T_5\\
0&0&0&-w&z&-x
\end{pmatrix}.
\]
The first three  components of the composition $\varphi\partial_3^*$, modulo $J$ are:
\begin{equation*}
[xq],
\quad
[T_1q],
\quad
\text{and}
\quad
-[T_2q].
\end{equation*}
They are all zero in $Q/J$ as $J=(x^2,\, xz,\, xw,\, T_1,\, T_2)$.
The fourth component is
\begin{equation*}
    -[w(y^3+cyq)]=-c[xw][y^2]-[yw(y^2-cp)]=-[y][wT_5]=0,
\end{equation*}
where the first equality follows from $q=xy-p$, the second equality follows as $[xw]=0$ and the last equality follows as $[wT_5]=0$ by \eqref{iden1}.
The fifth component is
\begin{equation*}
    [z(y^3+cyq)]=c[xz][y^2]+[yz(y^2-cp)]=[y][zT_5]=0,
\end{equation*}
where the second equality follows as $[xz]=0$ and the last equality follows as $[zT_5]=0$ by \eqref{iden1}.
Finally, the sixth component is
\begin{equation*}
  [T_5q-x(y^3+cyq)]=-[pT_5]+[xy][T_5-y^2+cp]+ [cy^2][x^2]=0, 
\end{equation*}
where the second equality follows as  $z$ divides $p$ and $[zT_5]=0$ by \eqref{iden1}. $[x^2]=0$, and $[T_5]=[y^2-cp]$.
 
To prove the inclusion $\Ker\varphi\subseteq\im\partial_3^*$, we consider $f$ and $g$ homogeneous polynomials in $Q$ such that 
\begin{equation}
\label{ker Iab}
    qf+(y^3+cyq)g\in  J.
\end{equation}
By Lemma \ref{J entries}, if  $f,g\in J$, then $\bigl(\begin{smallmatrix}f\\g\end{smallmatrix}\bigr)\in\im\partial_3^*$.  Moreover, by using the first and last three columns of $\partial_3^*$, we may assume that $f\in\kk[y,z,w]$ and $g\in\kk[y].$ In particular, there exists $e\in\kk$ such that $g=ey^n$ for some $n\geq 0.$
From \eqref{ker Iab} we deduce that $ey^{n+3}\in J+(q)\subseteq(x,\,z,\,w)$, hence $e=0$ and \eqref{ker Iab} becomes $qf\in J$, i.e.\ $f\in J\colon (q)$.
By Proposition \ref{J colon} we have $J\colon (q)=J+(x)$, thus $f\in J.$
We conclude that $\Ker\varphi=\im\partial_3^*$, thus the sequence in part (b) is exact. 
\end{proof}

\begin{setup}
\label{mu7-setup-c} 
We consider the ideals $I$ as Proposition in \ref{mu7-9 generators}.II that come from connected sums.  First, for each one of the cases IIa-IId,  we identify an ideal $J$ of $Q$ such that $J\subset I$. Second,  using the same terminology due to Brown \cite{Bro87}, we give a uniform description of the ideals $J$ that satisfy the conditions from Theorem \ref{mu7 structure}; see Proposition \ref{J res mu7-c}.

\begin{enumerate}[\bf a.]
\item[\bf IIa.] \textbf{Non-singular cubic.} Let $a\in \kk$ such that $a(a-1)(a^2-a+1)\not=0$ and set $b=(a^2-a+1)^{-1}$. We may write:
\begin{flalign*}
\hspace{0.5cm}
I&=J + \bigl(z^2+a(a+1)bzw+a^2bw^2,\ x^3-y^2z \bigr),\ \text{where}\\
J&=\bigl(xy,\, xz,\, xw,\, yw,\, a^2y^2+azw+(a+1)z^2\bigr)&   
\end{flalign*}
\item[\bf IIb.]  \textbf{Node-singularity.} We may write:
\begin{flalign*}
\hspace{0.5cm}
I&= J + (yw+z^2,\, x^3-w^3),\ \text{where}&\\
J&=(xy,\, xz,\, xw,\, y^2,\,yz-w^2)&
\end{flalign*}
\item[\bf IIc.]  \textbf{Conic and a secant line.}  We may write:
\begin{flalign*}
\hspace{0.5cm}
  I&= J + (z^2+w^2,\, x^3-yz^2),\ \text{where}&\\
  J&=(xy,\,xz,\, xw,\, y^2+w^2,\, zw)&
\end{flalign*}
\item[\bf IId.]   \textbf{Three lines not meeting at a point.} We may write:
\begin{flalign*}
\hspace{0.5cm}
    I&=J + (z^2,\, x^3-yzw),\ \text{where}&\\
    J&=(xy,\, xz,\, xw,\, y^2,\,w^2)&
\end{flalign*}
    \end{enumerate}
If we consider,
\begin{equation*}
  \arraycolsep=6pt
  \def\arraystretch{1.5}
  \begin{array}{c|cccc}
    \text{\scriptsize Case }&A&B&C&D\\
    \hline
    \text{\scriptsize \bf{IIa}}&y&0&a^{-1}z&-a^{-2}(a+1)z\\
    \text{\scriptsize \textbf{IIb}}&-w&-y&0&0\\
    \text{\scriptsize \textbf{IIc}}&z&0&w&0\\
   \text{\scriptsize \textbf{IId}}&w&0&0&0
  \end{array}
\end{equation*}
then the submaximal Pfaffians,   
$\{T_i\}_{1\leq i\leq 5}$, of the skew-symmetric matrix

\begin{equation}
  \label{matrix-c}
T=\begin{pmatrix}
0&0&A&B&0\\
0&0&C&D&y\\
-A&-C&0&y&z\\
-B&-D&-y&0&w\\
0&-y&-z&-w&0
\end{pmatrix},
\end{equation}
where $T_i$ is the Pfaffian of the matrix obtained from $T$ by removing row $i$ and column $i$ are:
\begin{align*}
  T_1&=wC-zD+y^2=
  \begin{cases} 
  y^2+a^{-2}(a+1)z^2+a^{-1}zw&\text{if {IIa}}\\
  y^2&\text{if {IIb, IId}}\\
  y^2+w^2&\text{if {IIc}}
  \end{cases}\\
  T_2&=wA-zB=\begin{cases} 
  yw&\text{if {IIa}}\\
  yz-w^2&\text{if {IIb}}\\
  zw&\text{if {IIc}}\\
  w^2&\text{if {IId}}
  \end{cases}\\
  T_3&=-yB=\begin{cases} 
  y^2&\text{if {IIb}}\\
  0&\text{otherwise}\\
  \end{cases}\\ 
  T_4&=-yA=\begin{cases} 
  -y^2&\text{if {IIa}}\\
  yw&\text{if {IIb}}\\
  -yz&\text{if {IIc}}\\
  -yw&\text{if {IId}}
  \end{cases}\\ 
  T_5&=BC-AD=\begin{cases} 
  a^{-2}(a+1)yz&\text{if {IIa}}\\
  0&\text{otherwise.}
  \end{cases}
\end{align*}
Moreover, the following equality holds
\begin{equation}
\label{defJ II}
    J=(xy,\,xz,\,xw,\,T_1,\,T_2).
\end{equation}
The following equalities are easy to check:
\begin{align}
\label{iden II}
&AT_1-CT_2+yT_4-zT_5=0,\\
&BT_1-DT_2+yT_3-wT_5=0,\, \text{and}\,\notag\\
&yT_2-zT_3+wT_4=0.\notag
\end{align}
\end{setup}

Based on a description from \cite[Theorem 4.4]{Bro87} we have the following result:

\begin{proposition}
\label{J res mu7-c} 
Adopt Setup \ref{mu7-setup-c}. Then the following assertions hold.
\begin{enumerate}[$(a)$]
\item  The ideal $J$ is a grade three perfect homogeneous  ideal with  a graded minimal free resolution of $Q/J$ over $Q$ given by  
\[
{\bG}\colon 0\to 
\begin{matrix}Q(-5)\\ \oplus\\ Q(-4)\end{matrix}
\xra{\partial_3}
\begin{matrix}Q(-4)\\ \oplus\\ Q^5(-3)\end{matrix}
\xra{\partial_2}Q^5(-2)\xra{\partial_1}Q,
\]
where
\begin{align*}
\partial_3&=
     \begin{pmatrix}
     x&0\\
     T_1&0\\
     -T_2&0\\
     T_3&-w\\
     -T_4&z\\
     T_5&-y
     \end{pmatrix},&
     \partial_2&=
     \begin{pmatrix}
 T_2&0&x&0&0&0\\
 T_1&-x&0&0&0&0\\
 0&A&C&0&-y&-z\\
 0&B&D&y&0&-w\\
 0&0&y&z&w&0
\end{pmatrix},\\
     &\hspace{2cm}\text{and}&\partial_1&=\begin{pmatrix} T_1&-T_2&-xw&xz&-xy\end{pmatrix}.
     \end{align*}
     \item The sequence 
$\begin{matrix}Q(-3)\\ \oplus\\ Q^5(-4)\end{matrix}
\xrightarrow{-\partial_3^*}
\begin{matrix}Q(-2)\\ \oplus\\ Q(-3)\end{matrix}
\xrightarrow{\varphi}I/J\to 0$
is exact, where 
\[
\varphi=\begin{cases}
    \begin{pmatrix}
\left[ z^2+(a^2b+ab)zw+a^2bw^2\right]&\left[x^3-y^2z\right]
    \end{pmatrix},& \text{if {IIa}}
    \\
    \begin{pmatrix}
    [z^2+yw]& [x^3+w^3]
    \end{pmatrix},&  \text{if {IIb}}
    \\
    \begin{pmatrix}
    [z^2+w^2]&[x^3-yz^2]
    \end{pmatrix},& \text{if {IIc}}
    \\
    \begin{pmatrix}
      [z^2]&[x^3-yzw]
    \end{pmatrix},& \text{if  {IId}}
\end{cases}
\]
with $b=(a^2-a+1)^{-1}$ in case IIa.
\end{enumerate}
\end{proposition}

\begin{proof} 
(a): 
It is clear that $J=\Coker\partial_1$. By the definitions of $T_1$ and $T_2$, and by \eqref{iden II} we get $\partial_1\partial_2=0$ and    $\partial_2\partial_3=0$, respectively. Thus $\partial_2^*\partial_1^*=0$ and $\partial_3^*\partial_2^*=0$. Therefore, both $\bG$ and $\Sigma^{-3}\bG^*$ are complexes. As in the proof of Proposition \ref{J res mu7}, we want to show that the grade of the ideal ${I_4(\partial_2)}$ is at least $2$ and the grades of the ideals  $I_1(\partial_1)$ and $I_2(\partial_3)$ are at least $3$.  

It is easy to see that the sequences $\{x^2z^2,\, wT_2^2\}$ in the cases IIa,b,d and $\{x^2y^2,\, zT_2^2\}$ in the case IIc are regular and contained in the ideal $I_4(\partial_2)$.  Thus, $\grade{}{I_4(\partial_2)}\geq 2$.

Now, we show that ${\grade{}{I_1(\partial_1)}=3}$ by proving  that the scheme defined by this ideal in projective space $\BP^3$ has dimension zero. This means that we can write this ideal as an intersection of ideals of grade three (primary ideals corresponding to the points in this scheme, counted with multiplicities). Next, we can apply \cite[Proposition 1.2.10 (c)]{BH93} and conclude that the ideal $J = I_1(\partial_1)$ has grade three. 
Since 
\[
J=(xy,\,xz,\,xw,\,T_1,\,T_2),
\]
with $T_1$ and $T_2$ as in Setup \ref{mu7-setup-c}, we see that if we set ${x\ne0}$, then ${y=z=w=0}$, which means that the point $[1:0:0:0]$ is in this scheme. Let us set ${x=0}$. In case IId, we must also have ${w=0}$, because ${T_2=w^2}$ (giving double points). So $T_1$ becomes a quadric in $y$ and $z$, defining two points, counted with multiplicities. In cases IIa, IIb, and IIc, we can consider both cases ${w=0}$ or ${w\ne0}$. If ${w=0}$ then $T_2$ vanishes and $T_1$ becomes again a quadric in $y$ and $z$. If ${w\ne0}$, then either $y$ or $z$ must be zero so that ${T_2=0}$, and $T_1$ becomes again a quadric in two variables. In any case, we see that the scheme is zero\h{dimensional}. A similar argument can be applied to show that $I_2(\partial_3)$ defines a zero\h{dimensional} scheme and therefore has grade three.

By the Acyclicity Criterion \cite[Theorem 1.4.13]{BH93} it follows that both complexes $\bG$  and $\Sigma^{-3}{\bG}^*$ are acyclic. Hence, $J$ is a perfect ideal of grade three.

(b): The surjectivity of  $\varphi$ follows from the description of the ideals $I$ and $J$.
Remark that 
\[\partial_3^*=
\begin{pmatrix}
x&T_1&-T_2&T_3&-T_4&T_5\\
0&0&0&-w&z&-y
\end{pmatrix}.
\]
To show the inclusion $\Ker\varphi\subseteq\im\partial_3^*$, we check that all six components of the composition $\varphi\partial_3^*$ are zero. The first component of the composition is zero as the first component of $\varphi$ is in the ideal $(z,w)$ and $[xz]=0=[xw].$ The second  and third components are zero as $[T_1]=0=[T_2]$.  To show that the rest of the components are zero, we consider each case separately.

\textbf{IIa.} In this case we have
\begin{align*}
\partial_3^*&=\begin{pmatrix}
x&T_1&-T_2&0&y^2&a^{-2}(a+1)yz\\
0&0&0&-w&z&-y
\end{pmatrix} \quad\text{and}\\
    \varphi&=\begin{pmatrix}
      \left[ z^2+(a^2b+ab)zw+a^2bw^2\right]&\left[x^3-y^2z\right]
    \end{pmatrix}.
\end{align*}
Since $[xw]=0=[yw]$, the forth component is: 
\begin{align*}
    -[x^3w]+[y^2zw]=-[x^2][xw]+[yz][yw]=0.
\end{align*}
 The fifth component is:
 \begin{align*}
  &[y^2z^2+(a^2b+ab)y^2zw+ a^2by^2w^2]+[x^3z-y^2z^2]\\
  &=[yw][(a^2b+ab)yz+ a^2byw]+[x^2][xz]\\
  &=0,
 \end{align*}
as $[yw]=0=[xz]$.
The  sixth component is:
\begin{align*}
 &[a^{-2}(a+1)(yz^3+(a^2b+ab)yz^2w+a^2byzw^2)]-[x^3y-y^3z]\\
 &=a^{-2}[yz][a^2y^2+azw+(a+1)z^2]-[x^2][xy]\\
 &\quad-a^{-1}[yw][z^2]+ a^{-2}(a+1)[yw][(a^2b+ab)z^2+a^2bzw]\\
 &=0,
\end{align*}
as $[a^2y^2+azw+(a+1)z^2]=[xy]=[yw]=0$.

\textbf{IIb.} In this case we have
\begin{align*}
\partial_3^*&=\begin{pmatrix}
x&T_1&-T_2&T_1&-yw&0\\
0&0&0&-w&z&-y
\end{pmatrix}\quad\text{and}\\
    \varphi&=\begin{pmatrix}
    [z^2+yw]& [x^3+w^3]
    \end{pmatrix}.
\end{align*}
The forth component  is:
\begin{align*}
    [T_1(z^2+yw)]-[x^3w+w^4]=[T_1][z^2+yw]-[x^2][xw]-[w^4]=0.
\end{align*}
as $[T_1]=[xw]=0$ and $[w^4]=[y^2z^2]=[y^2][z^2]=0.$
The fifth component is:
\begin{align*}
    -[yz^2w+y^2w^2]+[x^3z+zw^3]=-[zw][yz-w^2]-[y^2][w^2]+[x^2][xz]=0,
\end{align*}
as $[y^2]=[xz]=[yz-w^2]=0$. The sixth component is:
\begin{align*}
    -[x^3y+yw^3]=-[x^2][xy]+[yw][yz-w^2]-[y^2][zw]=0,
\end{align*}
$[xy]=[yz-w^2]=[y^2]=0$.

\textbf{IIc.} In this case we have
\begin{align*}
\partial_3^*&=\begin{pmatrix}
x&T_1&-T_2&0&yz&0\\
0&0&0&-w&z&-y
\end{pmatrix}\quad\text{and}\\
    \varphi&=\begin{pmatrix}
    [z^2+w^2]&[x^3-yz^2]
    \end{pmatrix}.
\end{align*}
The forth component is:
\begin{align*}
    -[x^3w-yz^2w]=-[x^2][xw]
    +[yz][zw]=0,
\end{align*}
as $[xw]=0=[zw]$.
The fifth component is: 
\begin{align*}
  [yz^3+yzw^2]+[x^3z-yz^3]=[yw][zw]+[x^2][xz]=0,  
\end{align*} 
as $[xz]=0=[zw]$. 
The sixth component is:
\begin{align*}
[-x^3y+y^2z^2]=-[x^2][xy]+[z^2][y^2+w^2]-[zw]^2=0,
\end{align*}
is zero as $[xy]=[y^2+w^2]=[zw]=0$. 

\textbf{IId.} In this case we have
\begin{align*}
\partial_3^*&=\begin{pmatrix}
x&T_1&-T_2&0&yw&0\\
0&0&0&-w&z&-y
\end{pmatrix}\quad\text{and}\\
    \varphi&=\begin{pmatrix}
    [z^2]&[x^3-yzw]
    \end{pmatrix}.
\end{align*}
The forth component is:
\begin{align*}
  [-x^3w+yzw^2]=-[x^2][xw]+[yz][w^2]=0,  
\end{align*}
 as $[xw]=0=[w^2]$.
 The fifth component is:
 \begin{align*}
     [yz^2w]+[x^3z-yz^2w]=[x^2][xz]=0,
 \end{align*} as $[xz]=0$.
The sixth component is:
\begin{align*}
    [-x^3y+y^2zw]=-[x^2][xy]+[y^2][zw]=0,
\end{align*}
as $[xy]=0=[y^2].$

If $q$ denotes the quadric generator of $I$ that is not in $J$ and $x^3-r$  the cubic generator of $I$, then we may write
\[\varphi =([q],[x^3-r]).\]
To prove the inclusion $\Ker\varphi\subseteq\im\partial_3^*$, we consider $f$ and $g$ homogeneous polynomials in $Q$ such that 
\begin{equation}
\label{qr}
    qf+(x^3-r)g\in J.
\end{equation}
By Lemma \ref{J entries}, if  $f,g\in J$, then  $\bigl(\begin{smallmatrix}f\\g\end{smallmatrix}\bigr)\in\im\partial_3^*.$ 
Moreover, by using the first column of $\partial_3^*$ for $f$ and last three columns of $\partial_3^*$ for $g$, we may assume that $f\in\kk[y,z,w]$ and $g\in\kk[x].$ In particular, there exists $e\in\kk$ such that  $g=ex^n$ for some $n\geq 0.$ 
The inclusion \eqref{qr} gives $ex^{n+3}\in J+(q,r)\subseteq (y,z,w),$ thus $e=0$ and so $g=0$. Hence the inclusion \eqref{qr} becomes $qf\in J$, which by Proposition \ref{colon II} gives $f\in J$.
The desired conclusion now follows.
\end{proof}

\begin{proof}[Proof of Theorem \ref{mu7 structure}.]
Let $J$ be an ideal defined as in Setups \ref{mu7-setup} and \ref{mu7-setup-c}. By Propositions \ref{J res mu7}(a) and \ref{J res mu7-c}(a), $J$  is a grade three perfect of $Q$ such that $Q/J$ has the resolution format $\mff_J=(1,5,6,2)$. 
The graded minimal free resolution of the $Q$-module $Q/J$ has the form:
\[
{\bG}\colon 0\to 
\begin{matrix}Q(-5)\\ \oplus\\ Q(-4)\end{matrix}
\xra{\partial_3}
\begin{matrix}Q(-4)\\ \oplus\\ Q^5(-3)\end{matrix}
\xra{\partial_2}Q^5(-2)\xra{\partial_1}Q,
\]
The dualizing module of the ring $Q/J$, up to isomorphism,  is given by  $\omega_{Q/J}=\Ext{Q}{3}{Q/J}{Q(-4)}=\HH{0}{\Sigma^{-3}\bG^*}(-4)=(\Coker\partial_3^*)(-4)$, see  \cite[Corollary 3.6.12; Example 3.6.15]{BH93}.
In Propositions \ref{J res mu7}(b) and \ref{J res mu7-c}(b), for each pair of ideals $(I,J)$, as in cases I and II respectively, we consider a concrete graded $Q$-homomorphism  $\varphi$ such that the sequence
\[\begin{matrix}Q(-3)\\ \oplus\\ Q^5(-4)\end{matrix}
\xrightarrow{-\partial_3^*}
\begin{matrix}Q(-2)\\ \oplus\\ Q(-3)\end{matrix}
\xrightarrow{\varphi}I/J\to 0
\]
is exact. It induces an isomorphism $\overline \varphi\colon\omega_{Q/J}(-3)\xra{} I/J$. The composition of the natural homogeneous embedding $I/J\hookrightarrow Q/J$ and the map $\overline\varphi$ gives a homogeneous embedding $\iota\colon\omega_{Q/J}(-3)\hookrightarrow  Q/J$. The desired conclusion now follows.\end{proof}

\begin{corollary}
\label{cor: mu7 structure} 
Let $\mathsf{k}$ be an algebraically closed field with $\cha\kk=0$ and $I$ a homogeneous Gorenstein  ideal of the ring  $Q=\kk[x,y,z,w]$, minimally generated by seven elements,  with $(x,y,z,w)^4\subseteq I\subseteq(x,y,z,w)^2$.
Then, a grade minimal free resolution of the $Q$\h{module} $Q/I$ has the form:
  \[
\bF\colon 0\to 
Q(-7)\xra{}
\begin{matrix}
Q(-4)\\
\oplus\\
Q^6(-5)
\end{matrix}
\xra{}
\begin{matrix}
Q^6(-3)\\
\oplus\\
Q^6(-4)
\end{matrix}
\xra{}
\begin{matrix}
Q^6(-2)\\
\oplus\\
Q(-3)
\end{matrix}
\xra{}
Q.
  \]
\end{corollary}

\begin{proof} From Propositions \ref{J res mu7}(a) and \ref{J res mu7-c}(a) we obtain that in each case, I and II respectively, a graded minimal free resolution $\bG$ of the $Q$-module $Q/J$ has the following form:
\[\bG\colon
0\to
\begin{matrix}
Q(-4)\\
\oplus\\
Q(-5)
\end{matrix}
\xra{}
\begin{matrix}
Q^5(-3)\\
\oplus\\
Q(-4)
\end{matrix}
\xra{}
Q^5(-2)\xra{}
Q
\]
The map 
$\varphi\colon Q(-2) \oplus Q(-3)\to I/J$ constructed in Propositions \ref{J res mu7}(b) and \ref{J res mu7-c}(b), in  each one of the cases I and II respectively, gives a homogeneous $Q$-homomorphism  $\iota\colon \omega_{Q/J}(-3)\to Q/J $; see the proof of Theorem \ref{mu7 structure}.
The map $\iota$ extends to a chain homomorphism of graded complexes $\overline\iota$:

\begin{equation*}
 \xymatrixrowsep{2pc}
 \xymatrixcolsep{1pc}
 \xymatrix{
\Sigma^{-3}\bG^*(-7)\colon 
\ar@{->}[d]^{\overline\iota}
&0\ar@{->}[r]
&Q(-7)\ar@{->}[r]\ar@{->}[d]^{\iota_3}
&Q^5(-5)\ar@{->}[r]\ar@{->}[d]^{\iota_2}
&{\begin{matrix}Q(-3)\\ \oplus\\ Q^5(-4)\end{matrix}
}\ar@{->}[r]\ar@{->}[d]^{\iota_1}
&{\begin{matrix}Q(-2)\\ \oplus\\ Q(-3)\end{matrix}}
\ar@{->}[d]^{\iota_0}
\\
\bG\colon
&0\ar@{->}[r]
&{\begin{matrix}Q(-4)\\ \oplus\\ Q(-5)\end{matrix}}
\ar@{->}[r]
&{\begin{matrix}Q^5(-3)\\ \oplus\\ Q(-4) \end{matrix}}
\ar@{->}[r]
&Q^5(-2)\ar@{->}[r] 
&Q
}
\end{equation*}
Set $\bF\colon =\cone\overline\iota$. By Lemma \ref{exact F}, $\bF$ is a graded  minimal free resolution of the $Q$-module $Q/I$ of the desired form.
\end{proof}

\section{Structure of Gorenstein ideals with 9 generators}
\label{Structure9gens}

The main result of this section is the following.
 \begin{theorem} 
\label{mu9 structure}
Let $\mathsf{k}$ be an algebraically closed field with $\cha\kk=0$ and $I$ a homogeneous Gorenstein  ideal of the ring  $Q=\kk[x,y,z,w]$, minimally generated by nine elements,  with $(x,y,z,w)^4\subseteq I\subseteq(x,y,z,w)^2$.
Then, there exists a grade three homogeneous perfect ideal $J$ of resolution format $(1,6,8,3)$, such that  $J\subset I$, and a homogeneous embedding $\iota$ of $\omega_{Q/J}(-3)$ into $Q/J$ such that the following sequence is exact:
\[0\rightarrow \omega_{Q/J}(-3)\xra{\iota} Q/J\xra{\pi} Q/I\rightarrow 0,\]
where $\pi$ is the canonical projection.
\end{theorem}

The proof of this theorem, given at the end of this section, uses the description given in Proposition \ref{mu7-9 generators} of all possible ideals $I$ minimally generated by nine elements, up to a linear change of variables, and Propositions \ref{J res mu9} and \ref{J res mu9-c}.

\begin{setup}
  \label{mu9-setup}  We consider the ideals $I$ as in Proposition \ref{mu7-9 generators}{III}. First, we identify an ideal $J$ of $Q$ such that $J\subset I$, for each one of the cases {IIIa-d}.
Second, we use a terminology due to Eagon-Northcott \cite{EN62}, that  gives a uniform description of the ideals $J$ that satisfy the conditions from Theorem \ref{mu9 structure}; see  Proposition \ref{J res mu9}.
  
  \begin{enumerate}[\bf a.]
  \item[\bf IIIa.] We may write
  \begin{flalign*}
\hspace{0.5cm}
      I&=J+ (y^2z,\, y^3,\, w^3),\ \text{where}&\\
      J&=(x^2,\, xz,\, xw,\, xy+yz-zw,\, yw-w^2,\, z^2).&
  \end{flalign*}

\item[\bf IIIb.] Let $a,b\in\kk$. We may write 
\begin{flalign*}
\hspace{0.5cm}
    I&=J+(xy^2-w^3,\, y^3,\, y^2z),\ \text{where}&\\
    J&=\bigl(x^2,\, xz,\, xw,\, xy-z^2,\, az^2-zw,\, ayz-yw+b(a^2-b)z^2+(b-a^2)w^2\bigr)&
\end{flalign*}
  \item[\bf IIIc.] Let $a\in\kk$.
  \begin{flalign*}
\hspace{0.5cm}
  I&=J+(y^3,\, y^2w,\, yw^2-aw^3),\ \text{where}&\\
  J&=(x^2,\,xz,\, xw,\, xy-zw,\, z^2,\, axy+yz-w^2).&
  \end{flalign*}
  \item[\bf IIId.]\textbf{Two-two connected sum.} We may write:
\begin{flalign*}
\hspace{0.5cm}
I&=J+(y^3,\,xy^2-zw^2,\,w^3),\ \text{where}&\\ 
J&=(x^2,\, xz,\, xw,\,yz,\,yw,\,z^2).&    
\end{flalign*}
  \end{enumerate}
If we consider, 
  \begin{equation*}
    \arraycolsep=5pt\def\arraystretch{1.5}
    \begin{array}{c|ccc}
      \text{\scriptsize Case} &A&B&C\\
    \hline
      \text{\scriptsize \textbf{IIIa}}&y-w&-y+w&-z\\
      \text{\scriptsize \textbf{IIIb}}
       &z&y+(a^2-b)w&ay+b(a^2-b)z\\
      \text{\scriptsize \bf IIIc}
       &z&w&y+aw\\
       \text{\scriptsize \bf IIId}
       &y&0&z\\
    \end{array}
  \end{equation*}
then, 
\begin{equation*}
J=(x^2, xz, xw, Az-Bx, Aw-Cx, Bw-Cz),
\end{equation*} i.e. $J$ is generated by the $2\times 2$ minors of the  matrix 
\begin{equation*}
\label{matrixMsetup6.2}
M=\begin{pmatrix}x&A&B&C\\0&x&z&w\end{pmatrix}.
\end{equation*}
\end{setup}

Recall that  $[-]$ denotes the operation of taking modulo $J$ of an element in $Q$. Based on a description from \cite{EN62} we have the following result:

\begin{proposition}
\label{J res mu9}
Adopt Setup  \ref{mu9-setup}. The following assertions hold.
\begin{enumerate}[$(a)$]
\item The ideal $J$ is a grade three perfect homogeneous  ideal with minimal free resolution of $Q/J$ over $Q$ given by 
\[
{\bG}\colon 0\to Q^3(-4)\xra{\partial_3}Q^8(-3)\xra{\partial_2}Q^6(-2)\xra{\partial_1}Q,
\]
where
\begin{align*}
\partial_3&=
\begin{pmatrix}
     C&w&0\\
     0&C&w\\
     -B&-z&0\\
     0&-B&-z\\
     A&x&0\\
     0&A&x\\
     -x&0&0\\
     0&-x&0
\end{pmatrix},&
\partial_2&=
\begin{pmatrix}
     -B&-z&-C&-w&0&0&0&0\\
     A&x&0&0&-C&-w&0&0\\
     -x&0&0&0&0&0&-C&-w\\
     0&0&A&x&B&z&0&0\\
     0&0&-x&0&0&0&B&z\\
     0&0&0&0&-x&0&-A&-x
\end{pmatrix},\\
&\hspace{2cm}\text{and}&\partial_1&=
\begin{pmatrix}
     x^2&xz&zA-xB&xw&wA-xC&wB-zC
\end{pmatrix}.
\end{align*}
\item The sequence  $Q^8(-4)\xrightarrow{-\partial_3^*}Q^3(-3)\xrightarrow{\varphi}I/J\to 0$ is exact, where\\

\[
\varphi=
\begin{cases}
\begin{pmatrix}
[w^3]&[y^2z]&[y^3-w^3]
\end{pmatrix},
&\ \text{\!\negthickspace\negthickspace\negthickspace\negthickspace if {IIIa}}\\  
\begin{pmatrix}
[w^3-xy^2]&[ac(w^3-xy^2)+y^2z]&[bc^2(w^3-xy^2)-y^3]
\end{pmatrix}
&\ \text{\!\negthickspace\negthickspace\negthickspace\negthickspace if {IIIb}}\\
\begin{pmatrix}[yw^2-aw^3]&[-y^2w]&[ay^2w+y^3]\end{pmatrix},
&\ \text{\!\negthickspace\negthickspace\negthickspace\negthickspace if {IIIc}}\\
\begin{pmatrix}[w^3]&[xy^2-zw^2]&[-y^3]\end{pmatrix},
&\ \text{\!\negthickspace\negthickspace\negthickspace\negthickspace if {IIId}}
\end{cases}
\]
where $c=a^2-b$ in case IIIb.
\end{enumerate}
\end{proposition}
\begin{proof}
(a): It is clear that $J=\Coker\partial_1$.  Remark that $\bG$ is the Eagon-Northcott complex, see \cite{EN62},  associated with the map given by the matrix $M$ in Setup \ref{mu9-setup}. 
Clearly, we have $\partial_1\partial_2=0$ and    $\partial_2\partial_3=0$, thus $\partial_2^*\partial_1^*=0$ and $\partial_3^*\partial_2^*=0$. Therefore, both $\bG$ and $\Sigma^{-3}\bG^*$ are complexes.
    
By a Macaulay2 calculation \cite{GS} one can check that the sequences: $\{x^5,z^5\}$ in cases IIIa,c,d and $\{x^5,az^5-z^4w\}$ in case IIIb are in the ideal $I_5(\partial_2)$. It is easy to check that they are both regular sequences.  Thus, $\grade{}{I_5(\partial_2)}\geq 2$.
Similarly, also by Macaulay2 calculation, the sequences: $\{x^3,\,z^3,\,yw^2-w^3\}$ in case IIIa, $\{x^3,\,z^3,\,a^2w^3-ayzw-bw^3+yw^2\}$ in case IIIb,
$\{x^3,\,z^3,\,xy^2-w^3\}$ in case IIIc,
and $\{x^3,\,z^3,\,yw^2\}$ in case IIId  belong to the ideal $I_3(\partial_3)$.  It is easy to check that they are all regular sequences.
Thus, $\grade{}{I_3(\partial_3)}\geq 3$.

Finally, we can use a similar argument to the one we use in the proof of Proposition \ref{J res mu7-c} to show that $J$ defines a zero\h{dimensional} scheme, and therefore $J=I_1(\partial_1)$ has grade three. It is immediate from the generators that this scheme includes the point $[0:1:0:0]$, so it is not empty. On the other hand, we know that $x$ vanishes, so we must also have ${zA=wA=wB=0}$. We can see that in each case this scheme has a finite number of points, so it is zero\h{dimensional}.

By the Acyclicity Criterion \cite[Theorem 1.4.13]{BH93} it follows that both complexes $\bG$  and $\Sigma^{-3}\bG^*$ are acyclic. Thus,  $J$ is a grade three perfect ideal.

(b):  It is clear that $\varphi$ is a surjective homomorphism. 
Considering each case separately, we prove the inclusion $\im\partial_3^*\subseteq\Ker\varphi$, i.e.\ we show that each one of the eight components of the composition $\varphi\partial_3^*$  is zero modulo $J$. To prove the other inclusion $\Ker\varphi\subseteq\im\partial_3^*,$ we consider a homogeneous element of $\Ker\varphi$, and prove that it is in the image of $\partial_3^*$.

By Lemma \ref{J entries} this holds if all the components of this element are in  the ideal $J$.

   \textbf{IIIa.}  In this case we have
   \begin{align*}
       \partial_3^*&=
   \begin{pmatrix}
    -z&0&y-w&0&y-w&0&-x&0\\
    w&-z&-z&y-w&x&y-w&0&-x\\
    0&w&0&-z&0&x&0&0
    \end{pmatrix}\quad\text{and}\\
    \varphi&=\begin{pmatrix}
[w^3]&[y^2z]&[y^3-w^3]
\end{pmatrix}.
   \end{align*}
   First, we prove that all the eight components of  $\varphi\partial_3^*$  are zero modulo $J$.
 The first component  is: 
 \[[-zw^3+y^2zw]=-[zw(w^2-y^2)]=[z(w+y)][yw-w^2]=0,\] 
 as $[yw-w^2]=0$. The second and third components are respectively:
 \begin{align*}
    [-y^2z^2+(yw-w^2)(y^2+yw+w^2)]&=0\quad\text{and}&
    [(yw-w^2)w^2-y^2z^2]&=0,
 \end{align*}
 as $[yw-w^2]=[z^2]=0$.
The fourth component is: 
 \begin{align*}
[y^3z-y^2zw-y^3z+zw^3]=-[z(y+w)][yw-w^2]=0,
   \end{align*} 
as $[yw-w^2]=0$.
The fifth component is: \[[w^2(yw-w^2)+xy^2z]=0,\]
as $[yw-w^2]=[xz]=0$.
The sixth component is:
\[[y^3z-y^2zw+xy^3-xw^3]=[y^2][xy+yz-zw]-[xw][w^2]=0,\]
as  $[xy+yz-zw]=[xw]=0$.
Finally, the seventh and eight components are zero as $[x^2]=[xz]=0.$ 

Next, to prove the inclusion $\Ker\varphi\subseteq\im\partial_3^*$    
we consider homogeneous polynomials $f,g,h$ in $Q$ of the same degree $n$ such that
\begin{equation*}
    w^3f+y^2zg+(y^3-w^3)h\in J.
\end{equation*}
Using columns two, four, and six of  $\partial_3^*$, we may assume $h\in\kk[y]$. Thus, from the inclusion  $J\subseteq(x,z,w)$
it follows $y^3h\in(x,z,w)$, so $h=0$. Using columns one, three, and eight of $\partial_3^*$, we may assume $g\in\kk[y]$ and using the columns five, seven, and eight, we may assume $f\in\kk[y,z]$. 
The equality $[w^2]=[yw]$ implies $[w^3]=[y^2w]$, thus 
\[[w^3f+y^2zg] = [y^2wf+y^2zg] = [y^2][wf+zg] = 0.\]

By Proposition \ref{colon JIII}(a) we have $J\colon(y)=J$, so  we get $[wf+zg] = 0$, and in particular ${n\ge1}$ because $J$ does not have any linear element. Moreover, since $[z^2]=0$, we may write $f=ay^n+by^{n-1}z$ for some $a,b\in\kk$. We may also write $g=cy^n$ for some $c\in\kk$. Thus, we have
\begin{align*}
   0=[wf+zg]&=[y^{n-1}][ayw+bzw+cyz]\\
          &=[y^{n-1}][ayw+b(xy+yz)+cyz]\\
          &=[y^n][aw+bx+(b+c)z]. 
\end{align*}
By Proposition \ref{colon JIII}(a) we get $aw+bx+(b+c)z\in J$ and as above this implies $a=b=c=0$, hence $f=g=0$. 

\textbf{IIIb.} If we set $c=a^2-b$, we have: 
   \begin{align*}
  \partial_3^*&=
   \begin{pmatrix}
      ay+bcz& 0&  -y-cw& 0& z& 0& -x& 0\\
       w&ay+bcz& -z& -y-cw& x& z& 0& -x\\
       0&  w& 0& -z&0& x& 0& 0
    \end{pmatrix}\quad\text{and}\\
    \varphi&=\begin{pmatrix}
[w^3-xy^2]&[ac(w^3-xy^2)+y^2z]&[bc^2(w^3-xy^2)-y^3]
\end{pmatrix}.
 \end{align*}
  First, we prove that all eight components of $\varphi\partial_3^*$  are zero modulo $J$.
By the definition of $J$  and Proposition \ref{colon JIII}(b) we have the following equalities in $Q/J$:
\begin{align*}
[(w^3-xy^2)z]&=0&[(w^3-xy^2)(y+cw)]&=[-xy^3]&[xy]&=[z^2]\\
[zw^2]&=0&[az^2-zw]&=0&[xz]&=0\\
&&[y^2w^2-a^2y^2z^2-c^2w^4]&=0&[yw^3]&=-c[w^4].
\end{align*}
We will use them in the computations that follow.
The first component is:
\begin{flalign*}
\hspace{0.5cm}
&[(w^3-xy^2)(ay+bcz)+(w^3-xy^2)acw +y^2zw]&\\
&=[a(w^3-xy^2)(y+cw)+y^2zw]=[-axy^3+y^2zw]&\\
&=[-ay^2z^2+y^2zw]=-[az^2-zw][y^2]=0.&
\end{flalign*}
The second component is: 
\begin{flalign*}
\hspace{0.5cm}
&\big[\big(ac(w^3-xy^2)+y^2z\big) (ay+bcz)+\big(bc^2(w^3-xy^2)-y^3\big)w\big]&\\
&=[a^2cyw^3-a^2cxy^3+ay^3z+bcy^2z^2+bc^2w^4-y^3w]&\\
&=[y^2][ayz-yw+bcz^2-cw^2]+[a^2cyw^3-a^2cxy^3+bc^2w^4+cy^2w^2]\\
&=[a^2cyw^3-a^2cy^2z^2+bc^2w^4+cy^2w^2]\\
&=c[y^2w^2-a^2y^2z^2-c^2w^4]+[a^2cyw^3+bc^2w^4+c^3w^4]\\
&=[-a^2c^2w^4+c^3w^4+bc^2w^4]\\
&=-c^2(a^2-b-c)[w^4]=0.&
\end{flalign*}
Indeed, the  first equality uses that  $[zw^2]=[xz]=[xw]=0$. The third equality uses $[ayz-yw+bcz^2-cw^2]=0$ and $[xy^3]=[y^2z^2]$.
The fifth equality uses $[y^2w^2-a^2y^2z^2-c^2w^4]=0$. The sixth equality uses $[yw^3]=-c[w^4]$,  and the last equality  uses $c=a^2-b$.  The third component is: 
\begin{flalign*}
\hspace{0.5cm}
&-[(w^3-xy^2)(y+cw)]-[acz(w^3-xy^2)+y^2z^2]&\\
&=[xy^3-y^2z^2]=[y^2][xy-z^2]=0.&
  \end{flalign*}
Indeed, the first equality uses  $-[(w^3-xy^2)](y+cw)=[xy^3]$ and $[z(w^3-xy^2)]=0$. The third equality uses $[xy-z^2]=0$.
The fourth component is: 
\begin{flalign*}
\hspace{0.5cm}
&-[\big(ac(w^3-xy^2)+y^2z\big)(y+cw)-bc^2z(w^3-xy^2)+y^3z]&\\
&=[acxy^3-y^3z-cy^2zw+y^3z]=[cy^2][a(xy-z^2)+(az^2-zw)]=0.&
  \end{flalign*}
  The first equality uses $-[(w^3-xy^2)(y+cw)]=[xy^3]$ and $[z(w^3-xy^2)]=0$. The third equality uses $[xy-z^2]=[az^2-zw]=0$. The fifth component is:
\begin{flalign*}
\hspace{0.5cm}
&[z(w^3-xy^2)+acx(w^3-xy^2)+xy^2z]&\\
&=[z(w^3-xy^2)]+ac[x(w^3-xy^2)]+[xz][y^2]=0.&
\end{flalign*}
The sixth component is: 
\begin{flalign*}
\hspace{0.5cm}
 &[acz(w^3-xy^2)+y^2z^2+bc^2x(w^3-xy^2)-xy^3]=-[y^2][xy-z^2]=0.&
\end{flalign*}
The seventh and the eight components are zero as  $[x(w^3-xy^2)]=[xz]=0.$ 
 
Next, to prove $\Ker\varphi\subseteq\im\partial_3^*$, 
we consider homogeneous polynomials $f,g,h$ in $Q$ of the same degree such that
\[
(w^3-xy^2)f+\big(ac(w^3-xy^2)+y^2z\big)g+ \big(bc(w^3-xy^2)-y^3\big)h\in J.
\]
By using the columns two, four, and six of $\partial_3^*$, we  reduce to the case $h\in\kk[y]$ and by using combinations of columns one, three, four, eight of $\partial_3^*$ we  reduce to the case $g=0$.
The inclusion
\begin{equation*}
\label{in ker 1}
    (w^3-xy^2) f+(bc(w^3-xy^2)-y^3)h\in J\subseteq(x,z,w)
\end{equation*}
implies that $y^3h\in (x,z,w)$, so $h=0$. By using columns five, seven, and eight of $\partial_3^*$ we reduce to the case $f\in\kk[y,w]$. The inclusion above becomes $(w^3-xy^2)f \in J$. By Proposition \ref{colon JIII}(b) we have $J\colon (w^3-xy^2)=J+(x,z)$, thus  $f\in J+(x,z)$. Since $f\in\kk[y,w]$ we get $f\in J,$ which completes the proof. 

\textbf{IIIc.}  In this case we have
\begin{align*}
   \partial_3^*&=\begin{pmatrix}
      y+aw& 0&  -w& 0& z& 0& -x& 0\\
       w&y+aw& -z& -w& x& z& 0& -x\\
       0&  w& 0& -z&0& x& 0& 0
    \end{pmatrix}\quad\text{and} \\
    \varphi &= \begin{pmatrix}[yw^2-aw^3]&[-y^2w]&[ay^2w+y^3]
    \end{pmatrix}.
    \end{align*}
First, we prove that all eight components of $\varphi\partial_3^*$  are zero modulo $J$. By the definition of $J$ and  Proposition \ref{colon JIII}(c) we have the following equalities in $Q/J:$
\begin{align*}
   [xy]&=[zw]& [w^2-axy-yz]&=0&[xw]&=0\\
   [zw^2]&=0& [(-yz+w^2)w]&=0&[w^4]&=0.
\end{align*}
We will use them in the computations that follow.  The first component is: \[[y^2w^2-ayw^3+ayw^3-a^2w^4-y^2w^2]=-a^2[w^4]=0.\]
The second component is: 
\[[-y^3w-ay^2w^2+ay^2w^2+y^3w]=0.\]
The third component is: 
\[[-yw^3+aw^4+y^2zw]=[-yw(-yz+w^2)]+a[w^4]=0.\]
The fourth component is: \[[y^2w^2-ay^2zw-zy^3]=[y^2][w^2-axy-yz]=0.\]
The fifth component is: 
\[[-yzw^2-azw^3-xy^2w]=-[y+aw][zw^2]-[y^2][xw]=0.\]
The sixth component is: 
\[[-y^2zw+axy^2w+xy^3]=[y^2][xy-zw]+[ay^2][xw]=0.\]
The seventh component is:  
\[[xyw^2-axw^3]=[xw][yw-aw^2]=0.\]
The eight component is: 
\[[xy^2w]=[y^2][xw]=0.\]

Next, to prove $\Ker\varphi\subseteq\im\partial_3^*$, we consider homogeneous polynomials $f,g,h$ in $Q$ of the same degree such that
\[(yw^2-aw^3)f+(-y^2w)g+(ay^2w+y^3)h\in J.\]
 By using the columns two, four, and six of $\partial_3^*$, we  reduce to the case $h\in\kk[y]$ and by using combinations of columns one, two, three, and eight of $\partial_3^*$ we reduce to the case $g=0$.
 By using the columns five, seven, and eight of $\partial_3^*$, we reduce to the case $f\in\kk[y,w].$ The inclusion
   \[(yw^2-aw^3)f+(ay^2w+y^3)h\in J\subseteq(x,z,w)\] implies that $h\in(x,y,z),$ hence $h=0$ and $(yw^2-aw^3)f\in J.$ By Proposition \ref{colon JIII}(c) we have $J\colon(w)=J+(x)$. Thus, since $f\in\kk[y,w]$, we obtain  $(y-aw)f\in J$.
   The inclusion $J\subset (x,z,w)$ implies that $f$ does not contain a pure power of $y$, so $w$ divides $f$. Using again $J\colon(w)=J+(x)$ and a descending induction on the degree of $f$ we obtain that $f\in J,$ which gives the desired conclusion.

\textbf{IIId:} In this case we have
  \begin{align*}
     \partial_3^*&= \begin{pmatrix}
      z& 0&  0& 0& y& 0& -x& 0\\
       w&z& -z& 0& x& y& 0& -x\\
       0&  w& 0& -z&0& x& 0& 0
    \end{pmatrix}\quad\text{and}\\
    \varphi&= \begin{pmatrix}[w^3]&[xy^2-zw^2]&[-y^3]\end{pmatrix}.
    \end{align*}
First, we prove that all eight components of $\varphi\partial_3^*$  are zero modulo $J$.
The first component is: 
\[[zw^3+xy^2w-zw^3]=[y^2][xw]=0.\]
The second component is: 
\[[xy^2z-z^2w^2-y^3w]=[y^2][xz]-[z^2][w^2]-[y^2][yw]=0.\]
The third component is: 
\[[-xy^2z+z^2w^2]=-[y^2][xz]+[z^2][w^2]=0.\]
The fourth component is: 
\[[y^3z]=[y^2][yz]=0.\]
The fifth component is: 
\[[yw^3 +x^2y^2-xzw^2]=[yw][w^2]+[x^2][y^2]-[xz][w^2]=0.\]
The sixth component is: 
\[[xy^3-yzw^2-xy^3]=-[yz][w^2].\] 
The seventh component is: 
\[-[xw^3]=-[xw][w^2]=0.\]
The eight component is: 
\[[-x^2y^2+xzw^2]=-[x^2][y^2]+[xz][w^2]=0.\]

Next, to prove $\Ker\varphi\subseteq\im\partial_3^*$, we consider homogeneous polynomials $f,g,h$ in $Q$ of the same degree such that 
\[w^3f+(xy^2-zw^2)g+(-y^3)h\in J.\]
 By using the columns two, four, and six of $\partial_3^*$, we  reduce to the case $h\in\kk[y]$ and by using combinations of columns one, three,  and eight of $\partial_3^*$ we reduce to the case $g\in\kk[y]$.
 By using the columns five, seven, and eight of $\partial_3^*$, we reduce to the case $f\in\kk[z,w].$ The inclusion $J\subseteq(x,z,w)$
  implies that $h\in(x,y,z),$ hence $h=0$. Since $J$ is a monomial ideal, we obtain $(xy^2)g\in J$ so $g=0$. The inclusion above becomes $w^3f\in J.$
  Using again that $J$ is monomial we obtain $f\in J$, which finishes the proof.  \end{proof}

\begin{setup}
\label{mu9-setup-c}
We consider the ideals $I$ as in Proposition \ref{mu7-9 generators}{IV}. First, we identify an ideal $J$ of $Q$ such that $J\subset I$, for each one of the cases {IVa-IVc}.
Second, we use a terminology due to \cite{EN62} that gives a uniform description of the ideals $J$  that satisfy the conditions from Theorem \ref{mu9 structure}; see Proposition
\ref{J res mu9-c}.

\begin{enumerate}[\bf a.]
\item[\bf IVa.] \textbf{Non-singular cubic.} 
Let $a\in\kk$ such that $a^2-a+1=0$.
We may write:
\begin{flalign*}
\hspace{0.5cm}
      I &= J+ ( 3x^3+w^3,\, y^3,\, z^3),\ \text{where}&\\
      J &=\bigl(xy,\, xz,\,  xw,\, yw,\, y^2-(a-1)zw-(2a-1)z^2,\, (2a-1)zw+(a-1)w^2\bigr).&
\end{flalign*}
\item[\bf IVb.] \textbf{Cusp singularity.} We may write:
\begin{flalign*}
\hspace{0.5cm}
I&= J+ (z^3,\, yz^2-w^3,\, x^3-w^3),\ \text{where}&\\
J&=(xy,\, xz,\, xw,y^2,\, yw,\, zw).&
\end{flalign*}

\item[\bf IVc.] \textbf{Conic and tangent line.}  We may write:
\begin{flalign*}
\hspace{0.5cm}
I&=J+(x^3-y^2z,\, y^3,\, y^2w),\ \text{where}&\\ 
J&=(xy,\, xz,\, xw,\, yz-w^2,\, z^2,\, zw).&
\end{flalign*}
\end{enumerate}

If we consider 
\begin{equation*}
    \arraycolsep=2.5pt\def\arraystretch{1.5}
    \begin{array}{c|ccc}
      \text{\scriptsize Case } &A&B&C\\
      \hline
      \text{\scriptsize \textbf{IVa}} &(2a-1)z+(a-1)w&y&0\\
      \text{\scriptsize \textbf{IVb}} &0&y&w\\
      \text{\scriptsize \textbf{IVc}} &w&0&z\\
    \end{array}
  \end{equation*}
then
\begin{equation*}
    J = (xy,\, xz,\, xw,\,  Az-By,\, Aw-Cy,\, Bw-Cz),
\end{equation*}
i.e. $J$ is generated by the $2\times 2$ minors of the matrix 
\begin{equation*}
 M=\begin{pmatrix}x&A&B&C\\0&y&z&w\end{pmatrix}.  
\end{equation*}
\end{setup}

Based on a description from \cite{EN62} we have the following result:

\begin{proposition}
\label{J res mu9-c}
 Adopt Setup  \ref{mu9-setup-c}. 
 The following assertions hold.
  \begin{enumerate}[$(a)$]
\item The ideal $J$ is a grade three perfect ideal with minimal free resolution of $Q/J$ over $Q$ given by
\[
{\bG}\colon 0\to Q^3(-4)\xra{\partial_3}Q^8(-3)\xra{\partial_2}Q^6(-2)\xra{\partial_1}Q,
\]
where
   \begin{align*}
     \partial_3&=
     \begin{pmatrix}
     C&w&0\\
     0&C&w\\
     -B&-z&0\\
     0&-B&-z\\
     A&y&0\\
     0&A&y\\
     -x&0&0\\
     0&-x&0
     \end{pmatrix},&
     \partial_2&=
     \begin{pmatrix}
     -B&-z&-C&-w&0&0&0&0\\
     A&y&0&0&-C&-w&0&0\\
     -x&0&0&0&0&0&-C&-w\\
     0&0&A&y&B&z&0&0\\
     0&0&-x&0&0&0&B&z\\
     0&0&0&0&-x&0&-A&-x
      \end{pmatrix},\\
     &&\partial_1&=
     \begin{pmatrix}
     xy&xz&zA-yB&xw&wA-yC&wB-zC
     \end{pmatrix}.
     \end{align*}
     \item The sequence 
$Q^8(-4)\xrightarrow{-\partial_3^*}Q^3(-3)\xrightarrow{\varphi}I/J\to 0$
is exact, where

$\varphi=
\begin{cases}
\begin{pmatrix}
      [-3z^3]&[-(2a-1)y^3]&[3x^3+w^3-3(2a-1)z^3]
    \end{pmatrix},&\text{if {IVa}}\\
\begin{pmatrix}
      [-yz^2-z^3+w^3]&[yz^2-w^3]&[-x^3+w^3]
    \end{pmatrix},&\text{if {IVb}}\\
\begin{pmatrix}
       [-y^3]&[y^2w]&[x^3-y^2z]
    \end{pmatrix},&\text{if {IVc.}}
\end{cases}$
\end{enumerate}
  \end{proposition}

\begin{proof}
It is clear that $J=\Coker\partial_1$.  Remark that $\bG$ is the Eagon-Northcott complex associated with the map given by the matrix $M$ in Setup \ref{mu9-setup}; see \cite{EN62}. 
It is easy to check that  $\partial_1\partial_2=0$ and    $\partial_2\partial_3=0$, thus $\partial_2^*\partial_1^*=0$ and $\partial_3^*\partial_2^*=0$. Therefore, both $\bG$ and $\Sigma^{-3}\bG^*$ are complexes.
    
There sequences
$\{xz+xw,\, yw,\, y^2-(a-1)zw-(2a-1)z^2\}$ in case IVa, 
$\{y^2,\, zw,\, xz+xw\}$ in case IVb, and  $\{xy,\, z^2,\, yz-w^2\}$ in case IVc are regular and contained in the ideal $I_1(\partial_1)$. To check this, we can use either the definition of regular sequence, in the easy cases, or an argument similar to the one we use in the proof of Proposition \ref{J res mu7-c} and show that the ideal generated by these sequences defines a zero\h{dimensional} scheme, and therefore has grade three. Thus the sequences must be regular. Hence, $\grade{}{I_1(\partial_1)}\geq 3$.
    
The sequences $\{x^2z^3,\, y^3w^2\}$ in case IVa, $\{y^5,\, z^3w^2\}$ in case IVb, and $\{x^2y^3,\, z^5\}$ in case IVc are regular and contained in the ideal $I_5(\partial_2)$. Hence, $\grade{}{I_5(\partial_2)}\geq 2$.
    
The sequences $\{x^2z+x^2w,yw^2,(y-z)\big(y^2-(a-1)zw-(2a-1)z^2\big) \}$ in case IVa, $\{y^3,\, z^2w,\, x^2z+x^2w\}$ in case IVb, and case $\{xy^2,\, z^3,\, w^3\}$ in case IVc are regular and contained in the ideal $I_3(\partial_3)$. Hence, $\grade{}{I_3(\partial_3)}\geq 3$.
   
By the Acyclicity Criterion \cite[Theorem 1.4.13]{BH93} it follows that both complexes $\bG$  and $\Sigma^{-3}{\bG}^*$ are acyclic. Therefore, $J$ is a grade three perfect ideal.
   
(b):  It is clear that in each case  $\varphi$ is a surjective homomorphism. 
  To prove the inclusion $\im\partial_3^*\subseteq\Ker\varphi$, we show that each one of the eight components of the composition $\varphi\partial_3^*$  is zero modulo $J$. To prove the other inclusion $\Ker\varphi\subseteq\im\partial_3^*,$ we consider a homogeneous element of $\Ker\varphi$  and show that it is in $\im\partial_3^*$. By Lemma \ref{J entries}, it is enough to show that all components of this homogeneous element are in the ideal $J$.
  
\textbf{IVa.} In this case we have:
\begin{align*}
\partial_3^*&=\begin{pmatrix}
      0 &0 &-y& 0&  (2a-1)z+(a-1)w& 0& -x& 0\\
      w& 0& -z& -y& y& (2a-1)z+(a-1)w& 0& -x\\
      0& w &0&-z& 0& y& 0&  0
    \end{pmatrix}  \\
   \varphi&=\begin{pmatrix}
      [-3z^3]&[-(2a-1)y^3]&[3x^3+w^3-3(2a-1)z^3]
    \end{pmatrix}.
\end{align*}
We show that each one of the eight components of the composition $\varphi\partial_3^*$  is zero modulo $J$. 
By the definition of $J$ and  Proposition \ref{colon JIV}(a) we have the following equalities in $Q/J:$
\begin{align*}
[y^3-(2a-1)yz^2]&=0 & [w^4-3(2a-1)z^3w]&=0&\\ [zw^3+3(a-1)z^3w]&=0 & [y^4+(a+1)z^3w+3z^4]&=0\\ [(2a-1)y^4-zw^3+3(2a-1)z^4]&=0.&&&
\end{align*}
The first component is: 
\[
[-(2a-1)y^3w]=-(2a-1)[y^2][yw]=0.
\]
The second component is: 
\[
[3x^3w+w^4-3(2a-1)z^3w]=3[x^2][xw]+[w^4-3(2a-1)z^3w]=0.
\]
Using that $(2a-1)^2=-3$, the third component simplifies as:
\[
[3yz^3+(2a-1)y^3z]=(2a-1)[y^3-(2a-1)yz^2][z]=0.
\]
The fourth component is:
\begin{align*}
  &[(2a-1)y^4-3x^3z-zw^3+3(2a-1)z^4]\\
&=[(2a-1)y^4-zw^3+3(2a-1)z^4]-3[x^2][xz]=0.  
\end{align*}
Using that $(2a-1)^2=-3$ and $-(2a-1)(a-1)=a+1$, the fifth component is:
\begin{align*}
&-[3(2a-1)z^4+3(a-1)z^3w+(2a-1)y^4]\\
&=-(2a-1)[y^4+(a+1)z^3w+3z^4]=0.
\end{align*}
Using that $(2a-1)^2=-3$ again, the sixth component is: 
\[
-[(2a-1)^2y^3z+3(2a-1)yz^3]=3[yz][y^2-(2a-1)z^2-(a-1)zw]=0
\]
The seventh component is: 
\[
[3xz^3]=3[xz][z^2]=0.
\]
The eight component is: 
\[
[(2a-1)xy^3]=(2a-1)[xy][y^2]=0.
\]

Next, to prove $\Ker\varphi\subseteq\im\partial_3^*$, we consider homogeneous polynomials $f,g,h$ in $Q$ such that 
\[
-3z^3f -(2a-1)y^3g+(3x^3+w^3-3(2a-1)z^3)h\in J.
\]
Using the columns two, four, and six we may reduce to the case when $h\in\kk[x]$. Using the columns one, three, five, and eight we may assume $g=0$. Using  column seven we may assume that  $f\in\kk[y,z,w]$. The inclusion  $J\subseteq (y,z,w)$ implies $h=0$ and hence we have $-3z^3f\in J$. By Proposition \ref{colon JIV}(a) we have $J\colon (z)=J+(x)$, so we get $f\in J$ which finishes the proof.

\textbf{IVb.} In this case we have:  
\begin{align*}
\partial_3^*&=\begin{pmatrix}
      w &0 &-y& 0&  0& 0& -x& 0\\
      w& w& -z& -y& y& 0& 0& -x\\
      0& w &0&-z& 0& y& 0&  0
    \end{pmatrix}\quad\text{and}\\
    \varphi&=\begin{pmatrix}
      [-yz^2-z^3+w^3]&[yz^2-w^3]&[-x^3+w^3]
    \end{pmatrix}.
\end{align*}
We show that each one of the eight components of the composition $\varphi\partial_3^*$  is zero modulo $J$. 
The first component is: 
\[[-yz^2w-z^3w+w^4+yz^2w-w^4]=-[z^2][zw]=0.\]
The second component is: 
\[[yz^2w-w^4-x^3w+w^4]=[yz][zw]-[x^2][xw]=0.\]
The third component is: 
\[[y^2z^2+yz^3-yw^3-yz^3+zw^3]=[y^2][z^2]-[yw][w^2]+[zw][w^2]=0.\]
The fourth component is:  \[[-y^2z^2+yw^3+x^3z-zw^3]=-[y^2][z^2]+[yw][w^3]+[x^2][xz]-[zw][w^2]=0.\]
The fifth component is: 
\[[y^2z^2-yw^3]=[y^2][z^2]-[yw][w^2]=0.\]
The sixth component is: 
\[[-x^3y+yw^3]=-[x^2][xy]+[yw][w^2]=0.\]
The seventh component is: 
\[[xyz^2+xz^3-xw^3]=[xy][z^2]+[xz][z^2]-[xw][w^2]=0.\]
The eight component is: 
\[[-xyz^2+xw^3]=-[xy][z^2]+[xw][w^2]=0.\]

Next, to prove $\Ker\varphi\subseteq\im\partial_3^*$, we consider homogeneous polynomials $f,g,h$ in $Q$ such that 
\[(-yz^2-z^3+w^3)f+ (yz^2-w^3)g+(-x^3+w^3)h\in J.\]
Using columns two, four, and six we may reduce to the case when $h\in\kk[x]$. Using columns one, three, five, and eight we may assume $g=0$. Using column seven we may assume that  $f\in\kk[y,z,w]$. The inclusion  $J\subseteq (y,z,w)$ implies $h=0$ and hence we obtain $(-yz^2-z^3+w^3)f\in J$. By Proposition \ref{colon JIV}(b) we have 
$J\colon(-yz^2-z^3+w^3)=J+(x)$, so we get $f\in J$, which finishes the proof.

\textbf{IVc.} In this case we have: 
\begin{align*}
\partial_3^*&=
\begin{pmatrix}
      z&0& 0& 0&w& 0& -x& 0\\
      w&z&-z& 0& y& w& 0& -x\\
      0&w& 0&-z& 0& y& 0&  0
    \end{pmatrix}\quad\text{and}\\
\varphi&=\begin{pmatrix}
      [-y^3]&[y^2w]&[x^3-y^2z]
    \end{pmatrix}.
\end{align*}
We show that each one of the eight components of the composition $\varphi\partial_3^*$  is zero modulo $J$. 
The first component is: 
\[[-y^3z+y^2w^2]=-[y^2][yz-w^2]=0.\]
The second component is: 
\[[y^2zw+x^3w-y^2zw]=[x^2][xw]=0.\]
The third component is: 
\[-[zy^2w]=-[y^2][zw]=0.\]
The fourth component is: 
\[[-x^3z+y^2z^2]=-[x^2][xz]+[y^2][z^2]=0.\]
The fifth component is: 
\[[-y^3w+y^3w]=0.\]
The sixth component is: 
\[[y^2w^2+x^3y-y^3z]=-[y^2][yz-w^2]+[x^2][xy]=0.\]
The seventh component is: 
\[[xy^3]=[xy][y^2]=0.\]
The eight component is:
\[-[xy^2w]=-[xy][yw]=0.\]

Next, to prove $\Ker\varphi\subseteq\im\partial_3^*$, we consider homogeneous polynomials $f,g,h$ in $Q$ such that 
\[(-y^3)f+y^2wg+(x^3-y^2z)h\in J.\]
Using the columns two, four, and six we may reduce to the case when $h\in\kk[x]$. Using the columns one, three, five, and eight we may assume $g=0$. Using the column seven we may assume that  $f\in\kk[y,z,w]$. The inclusion  $J\subseteq (y,z,w)$ implies $h=0$ and hence we have $-y^3f\in J$. By Proposition \ref{colon JIV}(c) we have $J\colon(y)=J+(x)$, hence $f\in J$ which finishes the proof.\end{proof}

\begin{proof}[Proof of Theorem \ref{mu9 structure}.]
Let $J$ be an ideal defined as in Setups \ref{mu9-setup} and \ref{mu9-setup-c}. By Propositions \ref{J res mu9}(a) and \ref{J res mu9-c}(a), $J$  is a grade three perfect ideal of $Q$ such that $Q/J$ has the resolution format $\ff_J=(1,6,8,3)$. The rest of the proof follows analogously to the proof of Theorem \ref{mu7 structure}.
\end{proof}

\begin{corollary}
\label{cor: mu9 structure} 
Let $\mathsf{k}$ be an algebraically closed field with $\cha\kk=0$ and $I$ a homogeneous Gorenstein  ideal of the ring  $Q=\kk[x,y,z,w]$, minimally generated by nine elements,  with $(x,y,z,w)^4\subseteq I\subseteq(x,y,z,w)^2$.
Then, a graded minimal free resolution of $Q/I$ has the form:
  \[
\bF\colon 0\to 
Q(-7)\xra{}
\begin{matrix}
Q^3(-4)\\
\oplus\\
Q^6(-5)
\end{matrix}
\xra{}
\begin{matrix}
Q^8(-3)\\
\oplus\\
Q^8(-4)
\end{matrix}
\xra{}
\begin{matrix}
Q^6(-2)\\
\oplus\\
Q^3(-3)
\end{matrix}
\xra{}
Q.
  \]
\end{corollary}
\begin{proof} From Propositions \ref{J res mu9}(a) and \ref{J res mu9-c}(a), in cases III and IV respectively,  we obtain that a graded minimal free resolution $\bG$ of the $Q$-module $Q/J$ that has the following form:
\[\bG\colon
0\to
Q^3(-4)
\xra{}
Q^8(-3)
\xra{}
Q^6(-2)\xra{}
Q.
\]
The map $\varphi\colon Q^3(-3)\to I/J$ constructed in Propositions \ref{J res mu9}(b) and \ref{J res mu9-c}(b), for each one of the cases III and IV, respectively, gives a homogeneous $Q$-homomorphism  $\iota\colon \omega_{Q/J}(-3)\to Q/J $; see the proof of Theorem \ref{mu9 structure}, and the hints for this proof in that of Theorem \ref{mu7 structure}.
The map $\iota$  extends to a chain homomorphism of graded complexes $\overline\iota$:

\begin{equation*}
 \xymatrixrowsep{2pc}
 \xymatrixcolsep{1pc}
 \xymatrix{
\Sigma^{-3}\bG^*(-7)\colon 
\ar@{->}[d]^{\overline\iota}
&0\ar@{->}[r]
&Q(-7)\ar@{->}[r]\ar@{->}[d]^{\iota_3}
&Q^6(-5)\ar@{->}[r]\ar@{->}[d]^{\iota_2}
&Q^8(-4)\ar@{->}[r]\ar@{->}[d]^{\iota_1}
&Q^3(-3)
\ar@{->}[d]^{\iota_0}
\\
\bG\colon
&0\ar@{->}[r]
&Q^3(-4)
\ar@{->}[r]
&Q^8(-3)
\ar@{->}[r]
&Q^6(-2)\ar@{->}[r] 
&Q
}
\end{equation*}
Set $\bF\colon =\cone\overline\iota$. By Lemma \ref{exact F},  $\bF$ is a  graded minimal free resolution of the $Q$-module $Q/I$ and  of the desired form.
\end{proof}
\section{Appendix}

In this section, we collect and prove several technical propositions that are used in the proofs of the main results of this paper.

\begin{proposition}
\label{J colon} Adopt Setup  \ref{mu7-setup}. 
The following equalities hold:
\[J\colon (z)=(x,\, T_1,\, T_2,\, T_5)=J\colon (p)
\quad\text{and}\quad  
J\colon (q)=J+(x).\]
\end{proposition}

\begin{proof} 
 Set $K=(x,\, T_1,\, T_2,\, T_5)$. 
 Using that $J =(x^2,\, xz,\, xw,\, T_1,\ T_2)$ and $zT_5\in(T_1,\,T_2)$ by \eqref{iden1}, we get $K\subseteq J\colon (z)$.   To prove  the reverse inclusion $J\colon (z)\subseteq K$, we consider  a homogeneous element $f$ in $Q$ such that $zf\in J$ and prove that $f$ is in $K$.
Since $x\in K$, we may assume that $f\in\kk[y,z,w]$.
Hence,   $zf\in J$ is equivalent to $zf\in (T_1,\, T_2)\kk[y,z,w]$. Therefore, there exist homogeneous polynomials $g_1,h_1\in\kk[y,z,w]$ and $g,h\in \kk[y,w]$ such that
${zf=(zg_1+g)T_1+(zh_1+h)T_2},$
which is equivalent to $z(f-g_1T_1-h_1T_2)=gT_1+hT_2$. As $T_1,T_2\in K$,  we assume that $g_1=h_1=0$, so 
\begin{equation}
\label{zf}
   zf=gT_1+hT_2. 
\end{equation}
Using the expressions for $T_1$ and $T_2$ we get
\begin{align*}
  zf=g(wC-zD)+h(wA-zB)=w(gC+hA)-z(gD+hB),
\end{align*}
which implies that  $z$ divides $gC+hA$.

\textbf{Ia.} Using the definitions of $A$ and $C$, in this case we obtain:
\begin{align*}
\quad gC+hA&=g((a^2+b)z-w)+hA=zg(a^2+b)+(-gw+hA).
\end{align*}
Using  $gC+hA\in(z)$, we get $-gw+hA\in(z)$.
Since $g,h,A\in\kk[y,w]$ it follows that $-gw+hA=0$. Since $\gcd(w,A)=1$, there exists $\ell\in \kk[y,w]$ such that $g=\ell A$ and  $h=\ell w$. The equation  \eqref{zf} becomes:
\begin{align*}
  zf&=\ell(AT_1+wT_2)\\
    &=\ell(AT_1-CT_2+(a^2+b)zT_2)\\
    &=\ell(zT_5+(a^2+b)zT_2)\\
    &=z\ell(T_5+(a^2+b)T_2),
\end{align*}
where the second equality uses the definition of $C$ and the third equality uses \eqref{iden1}. Dividing now by $z$, we obtain $f=\ell(T_5+(a^2+b)T_2)\in K$.

\textbf{Ib.} Using the definitions of $A$ and $C$, in this case we obtain:
\begin{align*}
gC+hA &=g(-az-w)+hA=-azg+(-wg+hA).
\end{align*}
Using $gC+hA\in(z)$, we get $-wg+hA\in(z)$. Since $g,h,A\in\kk[y,w]$ it follows that $-wg+hA=0$. Since $\gcd(w,A)=1$, there exists $\ell \in \kk[y,w]$ such that $g=\ell A$ and $h=\ell w$. The equation  \eqref{zf} becomes:
\begin{align*}
  zf&= \ell(AT_1+wT_2)\\ 
    &= \ell(AT_1-CT_2-azT_2)\\
    &= \ell(zT_5-azT_2)\\
    &= z\ell(T_5-aT_2),
\end{align*}
where the second equality uses the definition of $C$ and the third equality uses \eqref{iden1}. Dividing now by $z$, we obtain $f = \ell(T_5-aT_2)\in K$.

We conclude that the equality $K=J\colon (z)$ holds in both cases Ia and Ib. Next, we prove the equality $K=J\colon (p)$. Remark that it is enough to show: 
\begin{equation*}
    K=K\colon (w)\ \text{if Ia}
\quad
\text{and}
\quad
 K=K\colon (z)\ \text{if Ib.}
\end{equation*}
The inclusions $``\subseteq"$ clearly hold, so we prove the other inclusions.

\textbf{Ia.} Let $f$ be a homogeneous element in $Q$ such that $wf\in K$. We prove that $f\in K$. Since $x\in K$,  we assume that $f\in \kk[y,z,w]$. Arguing as above, we assume that  there exist homogeneous polynomials $g,h,k\in\kk[y,z]$ such that
\begin{equation}
\label{wf1}
 wf = gT_1+hT_2+kT_5.   
\end{equation}
By  \eqref{iden1}  we have  $zT_5 \in (T_1,T_2)\subseteq K$, thus we may further  assume that $k$  has no term involving $z$, so  $k=ey^n$ for some $e\in \kk$ and $n\geq 0$.
However, since $ey^{n+2}$ is the only term in the expression \eqref{wf1} containing $y$ as a pure power, we must have $e=0$, hence 
\begin{equation}
\label{wf2}
 wf = gT_1+hT_2.   
\end{equation}
Using the expressions of $T_1$ and $T_2$ we may further write:
\begin{align*}
  wf&=g(wC-zD)+h(wA-zB)\\
    &= w(gC+hA)-z(gD+hB)\\
    &= w\big(gC+hA-zh(a+b^2)\big)-z(gD+zh),
\end{align*}
hence $gD+zh\in(w)$. Since $g,D,h\in\kk[y,z]$ it follows that $gD+zh=0$. Since $\gcd(D,z)=1$, there exists $\ell\in \kk[y,z]$ such that
$g=z\ell$ and $h=-\ell D$. The equation  \eqref{wf2} becomes:
\begin{align*}
  wf &= \ell(zT_1-DT_2)\\
     &= \ell\big((z-B)T_1+BT_1-DT_2\big)\\
     &= \ell((a+b^2)wT_1+wT_5)\\
     &= w\ell((a+b^2)T_1+T_5),
\end{align*}
where the third equality uses the definition of $B$ and \eqref{iden2}. Dividing now by $w$, we get $f=\ell\big((a+b^2)T_1+T_5\big)\in K$.

\textbf{Ib.} Let $f$ be a homogeneous element in $Q$ such that $zf\in K$. We prove that $f\in K$. Since $x\in K$,  we assume that $f\in \kk[y,z,w]$. Arguing as above, we assume that  there exist homogeneous polynomials $g,h,k\in\kk[y,w]$ such that
\begin{equation}
\label{zf1}
zf = gT_1+hT_2+kT_5.
\end{equation}
By \eqref{iden2} we have  $wT_5 \in (T_1,T_2)\subseteq K$, thus we may further assume that $k=ey^n$ for some $e\in \kk$ and $n\geq 0$.
However, since $ey^{n+2}$ is the only term in equation \eqref{zf1} containing $y$ as a pure power, we must have $e=0$. Therefore, 
\begin{equation*}
\label{zf2}
zf = gT_1+hT_2.
\end{equation*}
By  the equality $K=J\colon (z)$ we conclude that $f\in K$.  

Finally, we prove the equality $J+(x)=J\colon(q)$. The inclusion $J+(x)\subseteq J\colon(q)$ is trivial.
To prove the other inclusion, it is enough to consider $f$ a homogeneous polynomial in $\kk[y,z,w]$ such that $qf\in J$ and to show that $f\in J$. 
There exist $e\in\kk,\ n\geq 0,$ and $h\in(z,w)\kk[y,z,w]$ such that  
$f=ey^n+h$.
\begin{align*}
    qf&=(xy-p)(ey^n+h)\\
      &=exy^{n+1}-ey^{n}p+h(xy-p).
\end{align*}
The inclusion \[qf\in J\subseteq (x^2,\, xz,\, xw,\, yz,\, yw,\, z^2,\, zw, w^2)\] gives $exy^{n+1}\in (x^2,\, xz,\, xw,\, yz,\, yw,\, z^2,\, zw,\, w^2)$, hence $e=0$ and we thus reduce to the case $f\in(z,w)\kk[y,z,w]$. 
The inclusion $qf\in J$ implies $pf\in J+(x)$. Using the equality $J\colon (p)=(x,T_1,T_2,T_5)$ we get $f\in (T_1,T_2,T_5)$. Since $T_1,T_2\in J$,  we may further  assume that $f\in (T_5)$. Since $T_5\in y^2+(z,w)\kk[y,z,w]$ and $f\in\kk[z,w]$, there exists a homogeneous polynomial $\ell\in(z,w)$ such that $f=T_5\ell.$ Using now \eqref{iden1} we conclude that $f$ belongs to $J$, as desired. 
\end{proof}

\begin{lemma}
\label{reg seq I}
Adopt Setup  \ref{mu7-setup}. The following are regular sequences in $Q$:
\begin{align*}
\{x^2, T_1, T_2\}\quad \{x^4,zw^2T_2^{\,2}\},\quad\text{and}\quad \{x^2, wT_1, zT_5\}.
\end{align*}
\end{lemma}
\begin{proof}
To check that $\{x^2, T_1, T_2\}$ is a regular sequence, we can see that since $x$ does not occur in $T_1$, the form $T_1$ cannot be a zero divisor in $Q/(x^2)$, and again since $x$ does not occur in $T_2$, for $T_2$ to be a zero divisor in ${Q/(x^2,T_1)}$ there must be ${g_1,g_2\in Q}$ such that $g_1T_2=g_2T_1$. But looking at $T_1$ in the ring $\kk[x,y,z][w]$, we can write 
\[
T_1=f_2w^2+f_1w+f_0=
\begin{cases}
-w^2+(a^2+b)zw+(yz-az^2),&\text{if\ {Ia}}\\
-w^2-azw+(yz+a^2z^2),&\text{if\ {Ib}.}
\end{cases}
\]
and we see that in both cases the irreducible element $z\in\kk[x,y,z]$ divides both $f_1$ and $f_0$, does not divide $f_2$, and $z^2$ does not divide $f_0$. So by Eisenstein Criterion, $T_1$ is irreducible, and therefore $T_1$ divides $g_1$; there are many references for this criterion, but for a simple one see \cite[Exercice 18.11]{Eis94}.

To check that $\{x^4,zw^2T_2^{\,2}\}$ is a regular sequence all we need is to note that $x$ does not occur in $zw^2T_2^{\,2}$ and argue as above.

Finally, to check that $\{x^2, wT_1, zT_5\}$ is a regular sequence we observe again that $x$ does not occur in $wT_1$, so this is not a zero\h{divisor} in $Q/(x^2)$, and, as above, since $x$ does not occur in $zT_5$, for $zT_5$ to be a zero divisor in ${Q/(x^2,wT_1)}$ there must be ${g_1,g_2\in Q}$ such that $g_1zT_5=g_2wT_1$. But we have seen that $T_1$ is irreducible, so we must have $T_1$ dividing $g_1$.
\end{proof}

\begin{proposition}
\label{colon II}
Adopt  Setup  \ref{mu7-setup-c}. The  following equality holds:
\[
J\colon (q)= J+(x),
\]
where $q$ denotes the quadric generator of $I$ that is not in $J$.
\end{proposition}

\begin{proof}  The inclusion $J+(x)\subseteq J\colon (q)$ is trivial.  Let $f$ be a homogeneous element in $Q$ such that  $qf\in J.$ We may assume that $f\in\kk[y,z,w]$ and in each one of the cases IIa-d we prove that $f\in J$. Recall that $[-]$ denotes the class of an element  modulo $J$.

\textbf{IIa.} We have 
$q=z^2 + ab(a + 1)zw + a^2bw^2$ 
and thus the inclusion 
\[
qf\in J\cap\kk[y,z,w]\subseteq(yw, y^2,zw,z^2)
\]
implies that $f$ does not have a pure power of $w$. Moreover, since $[yw]=0$ and $[y^2]=-[a^{-1}zw+a^{-2}(a+1)z^2]$, we may assume that $f=cyz^{n-1}+zh(z,w)$ for some $c\in\kk$ and $n\geq 0$.
Thus,
\begin{align*}
[qf]&=\bigl[cyz^{n+1}+zh(z,w)\bigl(z^2+ab(a+1)zw+{a^2}bw^2\bigr)\bigr]\\
&=z\bigl[cyz^{n}+h(z,w)\bigl(z^2+ab(a+1)zw+{a^2}bw^2\bigr)\bigr].
\end{align*}
Using that $J\cap \kk[y,z,w]=(yw,\, a^2y^2+azw+(a+1)z^2)\kk[y,z,w]$, we divide by $z$ and  get 
\begin{align*}
cyz^{n}+h(z,w)\big(z^2+ab(a+1)zw+{a^2}bw^2\big)\subseteq(yw, y^2,zw,z^2).
\end{align*}
The polynomial $h$ cannot have a pure power of $w$, so by induction on $n$, dividing by $z$ and using again that the simplified  expression on the left cannot have a pure power  of $w$,  we get $f=0$.

\textbf{IIb.}  We have $q=yw+z^2$ and thus the inclusion 
\[
qf\in J\cap\kk[y,z,w]\subseteq (y^2,yz, w^2)\kk[y,z,w]
\]
implies that $f$ does not have a pure power of $z$. Moreover, since $[y^2]=0$ and $[w^2]=[yz]$, we may assume that 
$f= byz^{n}+cyz^{n-1}w+dz^{n}w$ for some $b,c,d\in\kk$ and $n\geq 0$. 
Then 
\begin{align*}
qf&=(yw+z^2)(byz^{n}+cyz^{n-1}w+dz^{n}w)\\
&=byz^{n+2}+cyz^{n+1}w+dz^{n+2}w\\
 &\quad+y^2(bz^{n}w+cz^{n-1}w^2+dz^{n+1})-dyz^{n}(yz-w^2)
\end{align*}
and since $y^2,\, yz-w^2\in J$, the inclusion above also implies
\begin{equation*}
byz^{n+2}+cyz^{n+1}w+dz^{n+2}w\in J\cap\kk[y,z,w] \subseteq (y^2,yz,w^2)\kk[y,z,w],
\end{equation*}
hence $d=0$ and the inclusion becomes
\begin{equation*}
(byz+cyw)z^{n+1} \in J\cap\kk[y,z,w]= (y^2,yz-w^2)\kk[y,z,w].
\end{equation*}
By a descending induction argument on $n+1$ we get 
\[
byz+cyw\in (y^2,yz-w^2)\kk[y,z,w].
\]
We conclude that $b=c=0$, hence $f=0$.

\textbf{IIc.}  We have $q=z^2+w^2$ and thus the following inclusion 
\[
qf\in J\cap\kk[y,z,w]\subseteq (y^2,zw, w^2)\kk[y,z,w]
\]
imply that $f$ does not have a pure power of $z$. 
Using that $[zw]=0$ and $[y^2]=-[w^2]$, we may assume that $f=byz^n+cyw^n+dw^{n+1}$ for some $b,c,d\in\kk$ and $n\geq 0$.
\[
qf=(z^2+w^2)(byz^n+cyw^n+dw^{n+1})\in (y^2,zw, w^2)\kk[y,z,w] 
\]
implies that $byz^{n+2}\in (y^2,zw,w^2)$, hence $b=0$ and we get
\[
qf=(z^2+w^2)(cy+dw)w^{n}\in J\cap\kk[y,z,w]=( y^2+w^2,zw)\kk[y,z,w],
\]
which implies
\[
(cy+dw)w^{n+2}\in (y^2+w^2)\kk[y,z,w].
\]
Hence $c=d=0$, so $f=0$.

\textbf{IId.}  We have $q=z^2$ and thus the inclusion 
\[
qf=z^2f\in J\cap\kk[y,z,w]= (y^2,w^2)\kk[y,z,w]
\]
clearly implies that $f\in(y^2,w^2)\subseteq J$. 
\end{proof}

\begin{proposition} 
\label{colon JIII}
Adopt Setup \ref{mu9-setup}. 
\begin{enumerate}[$(a)$]
   \item If $J$ is as in case {IIIa}, then the following equality holds 
   \[J \colon (y)=J.\]
    \item If $J$ is as in case {IIIb}, then the following elements belong to $J$:
 \begin{align*}
     &z^3,\, z^2w,\, zw^2,\, axy^2-yzw,\,  x(w^3-xy^2),\,  z(w^3-xy^2),\\ 
     &(y+cw)(w^3-xy^2)+xy^3,\, yw^3+cw^4,\, y^2w^2-ay^2z^2-c^2w^4.
  \end{align*}
  Moreover, the following equalities hold   
  \[J\colon (z^2)=(x,z,w)\quad\text{and}\quad J\colon (w^3-xy^2)=J+(x,z).\]
     \item If $J$ is as in case {IIIc}, then the following elements belong to $J$:
\begin{align*}
   zw^2,\,  (-yz+w^2)w,\, w^4.
\end{align*}
Moreover, the following equality holds: \[J:(w)=J+(x).\]
\end{enumerate}
\end{proposition}
\begin{proof} 
Recall that $[-]$ denotes the class of an element  modulo $J$.

(a): The inclusion $J\subseteq J\colon(y)$ is trivial. Let $f$ be a homogeneous element in $Q$ such that $yf\in J$. From the inclusion $J\subset(x,z,w)$ it follows that $f$ does not have as monomial any free power of $y$. 
Using that $[x^2]=[xz]=[xw]=[z^2]=0$,  $[yz]=[-xy+zw]$, and $[yw]=[w^2]$, we may reduce to the case where $f=axy^n+bzw^n+cw^{n+1}$ with $a,b,c\in\kk$ and $n\geq 0$. 
Moreover, $[axy^{n+1}+bzyw^n+cyw^{n+1}]=0\iff [axy^{n+1}+bzw^{n+1}+cw^{n+2}]=0$.
If we make the linear change of variables $y'=y-w$, then $J=\bigl(x^2,\, xz,\, xw,\, z^2,\, (x+z)y',\, wy'\bigr)$, and we have 
\[
ax(y')^{n+1}+bzw^{n+1}+cw^{n+2}\in J\subseteq (x^2,\, xz,\, xw,\, z^2,\, xy',\, zy',\, wy').
\] 
Thus, $b=c=0$ and $ax(y')^{n+1}\in J$ implies now $a=0$. We conclude that $f\in J$.

(b): Set $c=a^2-b$. The listed elements belong to $J$ based on the following computations:
\begin{gather*}
\begin{aligned}
z^3&=(xz)y-(xy-z^2)z &axy^2-yzw&=a(xy-z^2)y+(az^2-zw)y\\
z^2w&=(xw)y-(xy-z^2)w     
&  x(w^3-xy^2)&=(xw)w^2-x^2y^2\\  
zw^2&=az^2w-(az^2-zw)w
&
z(w^3-xy^2)&=(zw^2)w-(xz)y^2
\end{aligned}\\
\begin{aligned}
(y+cw)(w^3-xy^2)+xy^3&=yw^3+cw^4-cxy^2w\\
&=-(ayz-yw+bcz^2-cw^2)w^2-cyw(xy-z^2)\\
&\quad +(ay+bcz)(zw^2)-(cy)(z^2w)\\
yw^3+cw^4&=-w^2(ayz-yw+bcz^2-cw^2)+(ay+bcz)(zw^2)
\end{aligned}\\
\begin{aligned}
[yw]&=[ayz+bcz^2-cw^2]
&[y^2w^2]&=[a^2y^2z^2+c^2w^4]+[b^2c^2z+2abcy][z^3]\\
&&&\quad-[2acy+2bc^2z][zw^2]\\
&&&=[a^2y^2z^2+c^2w^4].
\end{aligned}
\end{gather*}

The inclusion $(x,z,w)\subseteq J\colon(z^2)$ follows as $xz$, $z^3$, and $z^2w$ are all in $J$. 
Let $f$ be a homogeneous polynomial in $Q$ such that $z^2f\in J$. We may assume $f=ey^n$ for some $e\in\kk$ and $n\geq 1,$ so $ez^2y^n\in J$.
Using that $xy-z^2\in J$, this is equivalent to $exy^{n+1}\in J$.
If we consider the change of variables given by $w'=-w+az$, we get a simple presentation of the ideal $J$:
\[
J=\bigl(x^2,\,  xz,\,  xw',\,  xy-z^2,\,  zw',\,  yw'-c^2z^2-c(w')^2\bigr).
\]
There exist homogeneous polynomials $f_1,\dots, f_6$ in $Q$ such that 
\[
    exy^{n+1}=f_1x^2+f_2xz+f_3xw'+f_4(xy-z^2)+f_5zw'+f_6\big(yw'-c^2z^2-c(w')^2\big)
\]
Hence 
\[
exy^{n+1}-f_4(xy-z^2)-f_6\big(yw'-c^2z^2-c(w')^2\big)
\in (x^2,\, xz,\,  xw',\,  zw').
\]
Necessarily $f_4$ contains the term $ey^{n}.$  If $c=0$, since $z^2y^n\not \in (x^2,\, xz,\, xw',\, zw')$, we obtain $e=0.$  If $c\not =0$, then $f_6$ contains the term $c^{-2}ey^n$. It follows that $c^{-2}ey^{n+1}w'\in (x^2,\, xz,\, xw',\, zw')$, which also implies that $e=0.$  The desired conclusion now holds.

Using  $J+(x,z)=(yw+cw^2,\, x,\, z)$ and $(x^2,\, xz,\, xw,\, zw^2)\subseteq J$ it is clear that the inclusion $J+(x,z)\subseteq J\colon (w^3-xy^2)$ holds. Let $f$ be a homogeneous polynomial in $\kk[y,w]$ such that 
\begin{equation}
\label{incl IIIb}
    (w^3-xy^2)f\in J.
\end{equation}
Using $J\subseteq (yw+cw^2,x,z)$, we get $w^3f\in(yw+cw^2)\kk[y,w]$, hence
$f=(y+cw)g$ for some $g\in\kk[y,w].$
There exists $h\in\kk[y,w],\, e\in\kk$ and $n\geq 1$ such that $g=wh+ey^n.$
Hence $f=(yw+cw^2)h+e(y+cw)y^{n}$, and since $yw+cw^2\in J+(x,z)$, we may assume that $f=e(y+cw)y^{n}$.
The inclusion \eqref{incl IIIb} now becomes:
\[\big((w^3-xy^2)(y+cw)+xy^3\big)ey^{n}-exy^{n+3}\in J.\] 
By computations above, we have $(w^3-xy^2)(y+cw)+xy^3\in J$, so  $exy^{n+3}\in J$.
If we write $exy^{n+3}=e(xy-z^2)y^{n+2}+ez^2y^{n+2}$ and use that $xy-z^2\in J$ we get $ez^2y^{n+2}\in J$. Using the equality $J\colon (z^2)=(x,z,w)$ we get  $ey^{n+2}\in(x,z,w)$, so $e=0.$ The desired conclusion now follows.

(c): The listed elements belong to $J$ based on the following computations:
\begin{align*}
    zw^2&=(xw)y-(xy-zw)w\\
(-yz+w^2)w&=a(xw)y-(axy+yz-w^2)w\\
w^4&=a(xw)(yw)+y(zw^2)-(axy+yz-w^2)w^2.
\end{align*}
The implication $J+(x)\subseteq J\colon (w)$ is clear.
Let $f$ be a homogeneous element in $\kk[y,z,w]$ such that $fw\in J$. Clearly, the degree of $f$ is at least $2$ and as $z^2$ and $w^2-azw-yz$ are in $J$ and $zw\in J+(x)$, we may further assume that 
$f=by^n+cy^{n-1}w+dy^{n-1}z$ for some $b,c,d\in\kk$ and $n\geq 2$. 
Using the inclusion  $J\subseteq(yz,\, z^2,\, zw,\, w^2) +(x)$, we deduce that $b=0$, hence 
\begin{align*}
    wf&=cy^{n-1}w^2+dzy^{n-1}w=cy^{n-1}(w^2-yz-azw)+y^{n-1}z\big(cy+(ca+d)w\big).
\end{align*}
Using again that $w^2-yz-azw\in J$, we get
$y^{n-1}z\big(cy+(ca+d)w\big)\in J$. By the inclusion $J\cap \kk[y,z,w]\subseteq (z^2,zw,yz-w^2)\kk[y,z,w]$ we get $c=0$ and thus $f=dzy^{n-1}.$ 
By assumption $fw=dzy^{n-1}w\in J$ and by definition of $J$ we have $xy-zw\in J$, therefore $dxy^{n}\in J$. 
Remark that the last generator of $J$ can be replaced by $azw+yz-w^2$.
There exist homogeneous polynomials $f_1,\dots ,f_6\in Q$ such that 
\[
    dxy^{n}=f_1x^2+f_2xz+f_3xw+f_4(xy-zw)+f_5z^2+f_6(azw+yz-w^2).
\]
Hence 
\[
dxy^n-f_4(xy-zw)-f_6(azw+yz-w^2)\in(x^2,\, xz,\, xw,\, z^2).
\]
Necessarily, $f_4$ contains the term $dy^{n-1}$. If $a=0$, since $y^{n-1}zw\not\in(x^2,\, xz,\, xw,\, z^2),$ we get $d=0$. If $a\not=0,$ then $f_6$ must contain the term $a^{-1}dy^{n-1}$. Using that $y^{n-1}w^2\not\in(x^2,\, xz,\, xw,\, z^2),$ we also get $d=0$. The desired conclusion now holds.
\end{proof}

\begin{proposition} 
\label{colon JIV}
Adopt Setup \ref{mu9-setup-c}.
\begin{enumerate}[$(a)$]
\item If  $J$ is as in case {IVa}, then it contains the  following elements:
\begin{align*}
&y^3-(2a-1)yz^2,\, w^4-3(2a-1)z^3w,\, zw^3+3(a-1)z^3w, \\
&y^4+(a+1)z^3w+3z^4,\,  (2a-1)y^4-zw^3+3(2a-1)z^4.
\end{align*}
Moreover, it satisfies the following equality 
\[
J\colon (z)=J+(x).
\]
\item If $J$ is as in case {IVb}, then it satisfies the following equality:
\[
J\colon (-yz^2-z^3+w^3)=J+(x).
\]
\item If $J$ is as in case  {IVc}, then it satisfies the following equality:
\[
J\colon (y) = J+(x).
\]
\end{enumerate}
\end{proposition}

\begin{proof}
(a): 
First, note that since ${a^2-a+1=0}$, we also have the  following equalities:
\[
-a^3=1,\,
(a-2)^2=-3(a-1),\, 
(a-2)^3=3(2a-1),\, 
(2a-1)^2=-3.
\]
Using the following notation for some of the generators of $J$:
\[
f_1=yw,\, f_2=y^2-(a-1)zw  -(2a-1)z^2,\,\text{and}\, f_3=-(2a-1)zw-(a-1)w^2,
\]
the listed elements belong to $J$ based on the following computations:
\begin{align*}
y^3-(2a-1)yz^2&=(a-1)zf_1+yf_2\\
w^4-3(2a-1)z^3w&=(a+1)ywf_1-(a+1)w^2f_2+(aw^2+3z^2)f_3\\
zw^3+3(a-1)z^3w&= -aywf_1+aw^2f_2-(a+1)z^2f_3\\
y^4+(a+1)z^3w+3z^4&=(a-1)yzf_1+\big((2a-1)z^2+y^2\big)f_2\\
(2a-1)y^4-zw^3+3(2a-1)z^4&=\big(ayw-(a+1)yz\big)f_1\\
&\quad -\big(3z^2-(2a-1)y^2+aw^2\big)f_2\\
&\quad +(a+1)z^2f_3.
\end{align*}
The first equality follows by multiplying $f_2$ by $y$ and then by using the expression for $f_1$.
The fourth equality follows from the first by multiplying by $y^2$ and then using that $y^2=f_1+(a-1)zw+(2a-1)z^2$. The fifth equality follows by multiplying the fourth equality by $2a-1$ and then using the third equality.

Since $xz\in J$, the inclusion $J+(x)\subseteq J\colon (z)$ holds trivially. Let $f$ be a homogeneous polynomial in $\kk[y,z,w]$ such that $fz\in J.$
The inclusion
\[J\cap\kk[y,z,w]\subseteq \big(yw, y^2-(a-1)zw-(2a-1)z^2, (2a-1)zw+(a-1)w^2\big)\] implies that there exist homogeneous polynomials $g_1,g_2,g_3\in\kk[y,z,w]$ such that 
\[fz=ywg_1 + \big(y^2-(a-1)zw-(2a-1)z^2\big)g_2 +\big((2a-1)zw+(a-1)w^2\big)g_3.\]
Without loss of generality, we may assume that $g_1,g_2,g_3\in\kk[y,w]$. Therefore,  by taking $z=0$, we get the following equality in $\kk[y,w]$:
\[0=ywg_1 + y^2g_2 +(a-1)w^2 g_3.\]
It follows that $g_2=wh_2$ and $g_3=yh_3$ for some $h_2,h_3\in\kk[y,w].$
Moreover, we get 
\begin{align*}
fz&=\big(-(a-1)zw-(2a-1)z^2\big)wh_2 +(2a-1)zwyh_3, \ \text{so}\\
  f&=-\big((2a-1)zw+(a-1)w^2\big)h_2 +(2a-1)ywh_3\in J.
\end{align*}
The desired conclusion now follows.

(b): Since $xy,\, xz,\, xw \in J$, the inclusion $J+(x)\subseteq J\colon (-yz^2-z^3+w^3)$ clearly holds. 
Let $f$ be a homogeneous polynomial in  $\kk[y,z,w]$ such that $f(-yz^2-z^3+w^3)\in J$.
It is clear that $f$ does not contain pure powers of $z$ or $w$. Modulo $J$ we may assume $f=\alpha yz^n$  for some $\alpha\in\kk$ and $n\geq 0$. If $\alpha\not=0$, then $yz^{n+3}\in J$, which is a contradiction. Hence $f=0$, and the desired conclusion follows.

(c): Since $xy\in J$, the inclusion $J+(x)\subseteq J\colon (y)$ clearly holds.  Let $f$ be a homogeneous polynomial in  $\kk[y,z,w]$ such that $fy\in J.$
The inclusion $J\subseteq(xy, xz, xw, zw, z^2, yz, w^2)$ shows that $f$ does not contain a pure power of $y$, a pure power of $w$, or terms of the form $y^nw$  with $n\geq 0$. Modulo the ideal $J$ we may further reduce to the case when $f =\alpha z$ for some $\alpha \in \kk$.
If $\alpha\not=0$, then $yz\in (zw, z^2, yz-w^2)$, which is a contradiction. Hence $f=0$, and the desired conclusion follows. 
\end{proof}

\section*{Acknowledgements}
We thank Lars W. Christensen, Tony Iarrobino, Ela Celikbas, and Jai Laxmi for helpful conversations on the topic of this paper. We thank Giulio Caviglia for sharing his Master Thesis \cite{Cav00} with us; it was inspiration for our project. We also thank Andrew Kustin for sharing the preprint \cite{EKK17} with us and for helpful conversations.

P. Macias Marques was partially supported by CIMA -- Centro de Investiga\c{c}\~{a}o em Matem\'{a}tica e Aplica\c{c}\~{o}es, Universidade de \'{E}vora, project {\scriptsize UIDB/04674/2020} (Funda\c{c}\~{a}o para a Ci\^{e}ncia e Tecnologia).  
Part of this work was done while P.M.M. was visiting Northeastern University, he wishes to thank Tony Iarrobino, Oana Veliche, and the Mathematics Department for their hospitality.

J. Weyman was supported by the grants:
{\scriptsize MAESTRO  NCN - UMO-2019/34/A/ST1/00263}
Research in Commutative Algebra and Representation Theory and 
{\scriptsize NAWA POWROTY-PPN/PPO/2018/1/00013/U/00001}
Applications of Lie algebras to Commutative Algebra.


\providecommand{\bysame}{\leavevmode\hbox to3em{\hrulefill}\thinspace}
\providecommand{\MR}{\relax\ifhmode\unskip\space\fi MR }
\providecommand{\MRhref}[2]{%
  \href{http://www.ams.org/mathscinet-getitem?mr=#1}{#2}
}
\providecommand{\href}[2]{#2}

\end{document}